\documentclass[reqno,11pt]{amsart}
\usepackage{amscd,amssymb}
\usepackage{amsmath}
\usepackage{mathtools}
\usepackage{color}
\usepackage{hyperref}
\usepackage{comment}
\usepackage{slashed}
\usepackage{marginnote}
\usepackage{enumitem}
\usepackage{tikz}
\usetikzlibrary{matrix,arrows,decorations.pathmorphing}

\numberwithin{equation}{section}

\theoremstyle{plain}
\newtheorem{theorem}[subsection]{Theorem}
\newtheorem{lemma}[subsection]{Lemma}
\newtheorem{corollary}[subsection]{Corollary}
\newtheorem{proposition}[subsection]{Proposition}
\newtheorem{conjecture}[subsection]{Conjecture}

\newtheorem{definition}[subsection]{Definition}

\newtheorem{thmx}{Theorem}
\newtheorem{corx}[thmx]{Corollary}

\theoremstyle{definition}

\newtheorem{notation}[subsection]{Notation}

\theoremstyle{remark}
\newtheorem{remark}[subsection]{Remark}

\usepackage{amsmath,amsfonts,latexsym}

\newcommand{\tensor}{\otimes}

\DeclareMathOperator{\Diff}{Diff}

\DeclareMathOperator{\dom}{dom}

\DeclareMathOperator{\id}{id}

\DeclareMathOperator{\supp}{supp}   
\DeclareMathOperator{\End}{End}    

\DeclareMathOperator{\pr}{pr}

\DeclareMathOperator{\ind}{ind}

\DeclareMathOperator{\scal}{scal}  

\newcommand{\forget}[1]{}

\allowdisplaybreaks[2]

\newcommand{\RR}{\mathbb{R}}
\newcommand{\CC}{\mathbb{C}}
\newcommand{\ZZ}{\mathbb{Z}}

\newcommand{\Id}{\operatorname{Id}}

\newcommand{\Dom}{\operatorname{Dom}}

\usepackage{anysize}

\begin{document}
\title[Callias-type operators in $C^\ast$-algebras and PSC on noncompact manifolds]{Callias-type operators in $C^\ast$-algebras and positive scalar curvature on noncompact manifolds.}

\author{Simone Cecchini}
\address{Mathematisches Institut,
Georg-August-Universit\"at G\"ottingen, 37073 G\"ottingen, Germany}

\email{cecchini@mathematik.uni-goettingen.de}


\begin{abstract} 
A Dirac-type operator on a complete Riemannian manifold is of Callias-type if its square is a Schr\"{o}dinger-type operator with a potential uniformly positive outside of a compact set.
We develop the theory of Callias-type operators twisted with Hilbert $C^\ast$-module bundles and prove an index theorem for such operators. 
As an application, we derive an obstruction to the existence of complete Riemannian metrics of positive scalar curvature on noncompact spin manifolds in terms of closed submanifolds of codimension one.
In particular, when $N$ is a closed spin manifold, we show that if the cylinder $N\times\RR$ carries a complete metric of positive scalar curvature, then the (complex) Rosenberg index on $N$ must vanish.
\end{abstract}


\maketitle

\section{Introduction}

An important topic in differential geometry in recent decades is whether a given smooth manifold admits a Riemannian metric of positive scalar curvature. 
On closed spin manifolds, the most powerful obstructions to the existence of such metrics are based on the index theory for the spin-Dirac operator.
Indeed, the Bochner-Lichnerowicz formula \cite{Lic63} implies that, on a closed spin manifold $N$ with positive scalar curvature, the spin-Dirac operator is invertible and hence its index must vanish. 

Rosenberg (\cite{Ros83}, \cite{Ros86a}, \cite{Ros86b}) refined this obstruction by using Dirac operators twisted with flat Hilbert $C^\ast$-module bundles of finite type.
Let $\pi$ be the fundamental group of $N$ and let $C^\ast_\RR(\pi)$ be the \emph{real maximal} group $C^\ast$-algebra of $\pi$.
By twisting the spin-Dirac operator on $N$ with the canonical flat Hilbert $C^\ast_\RR(\pi)$-bundle over $N$, one obtains the Rosenberg index obstruction $\alpha_\RR(N)\in KO_\ast(C^\ast_\RR(\pi))$.
It was conjectured  that this obstruction gives a complete characterization of the existence of metrics of positive scalar curvature on closed spin manifolds.


\begin{conjecture}[Gromov-Lawson-Rosenberg]\label{C:G-L-R}
Let $N$ be a closed spin manifold of dimension at least $5$.
Then $N$ admits a metric of positive scalar curvature if and only if $\alpha_\RR(N)=0$.
\end{conjecture}


In the celebrated work~\cite{Sto92}, Stolz proved this conjecture when $N$ is simply connected.
Since then, many other cases have been proved.
On the other hand, the conjecture is not always true by the counterexample found by Schick in dimension $5$ (cf.~\cite{Sch98}).
For a comprehensive discussion  of this topic, we refer to the survey papers~\cite{Ros07} and \cite{Sto01}.

The study of complete metrics of positive scalar curvature on a noncompact manifold $M$ is more complicated.
In the case when $M$ is a cylinder with compact base, Rosenberg and Stoltz proposed the following conjecture.


\begin{conjecture}\label{C:R-S}\emph{(\cite[Conjecture~7.1]{RS94})}
Let $N$ be a closed manifold.
If $N\times\RR$ admits a complete metric of positive scalar curvature, then $N$ admits a metric of positive scalar curvature.
\end{conjecture}


When $N$ is enlargeable, this conjecture holds by results of Gromov and Lawson (see \cite[Corollary~6.13 and Theorem~7.5]{GL83}).
In \cite[Section~7]{RS94}, Stolz and Rosenberg proved many other cases by using the minimal surface technique.
When the manifold $N$ has a spin structure, it is possible to use Dirac obstructions on $N$ to construct obstructions to the existence of a complete metric of positive scalar curvature on $N\times \RR$.
In particular, the following theorem holds.


\begin{theorem}\label{T:Ros07}\emph{(Rosenberg, \cite[Theorem~3.4]{Ros07})}
Let $N$ be a closed  spin manifold.
If $N\times\RR$ admits a complete metric of \emph{uniformly} positive scalar curvature, then $\alpha_\RR(N)= 0$.
\end{theorem}


Notice that this theorem is a first step toward connecting Conjecture~\ref{C:G-L-R} with Conjecture~\ref{C:R-S}.
Suppose, in fact, that $N$ is a closed spin manifold satisfying Conjecture~\ref{C:G-L-R}.
If $N\times\RR$ admits a complete metric of \emph{uniformly} positive scalar curvature, then Theorem~\ref{T:Ros07} implies that $N$ admits a metric of positive scalar curvature.
A natural question to ask is whether Theorem~\ref{T:Ros07} holds under the weaker hypothesis that $N\times\RR$ admits a complete metric of (not necessarily uniformly) positive scalar curvature.

In the present paper, we work with complex $C^\ast$-algebras instead of real ones.
Let $C^\ast_\CC(\pi)$ be the (minimal or maximal) complex group $C^\ast$-algebra associated to the fundamental group $\pi$ of $N$.
By twisting the spin-Dirac operator on $N$ with the canonical flat Hilbert $C^\ast_\CC(\pi)$-bundle of finite type over $N$, we obtain an index obstruction $\alpha_\CC(N)\in K_\ast(C^\ast_\CC(\pi))$.
The first main result is the following theorem.


\begin{thmx}\label{T:Theorem A}
Let $M$ be a connected spin manifold without boundary and let $N\subset M$ be a closed connected codimension one submanifold with trivial normal bundle.
Suppose the inclusion $i:N\hookrightarrow M$ induces an injective homomorphism $i_\ast:\pi_1(N)\rightarrow\pi_1(M)$.
If $M$ admits a complete metric whose scalar curvature is nonnegative and strictly positive in a collar neighborhood of $N$, then $\alpha_\CC(N)= 0$.
\end{thmx}


\begin{remark}
Suppose $M$ is closed and $N$ is a submanifold with trivial normal bundle and such that the induced homomorphism $\pi_1(N)\rightarrow\pi_1(M)$ is injective.
Recently, Hanke, Pape and Schick~\cite{HPS15} proved that, when $N$ has codimension two and $\pi_2(M)=0$, the Rosenberg index of $N$ is an obstruction to the existence of metrics of positive scalar curvature on $M$.
When $N$ has codimension one, Zeidler~\cite{Zei16} proved that if $\alpha_\CC(N)\neq 0$, then $\alpha_\CC(M)\neq 0$ and $M$ cannot carry a metric of positive scalar curvature.
\end{remark}


\begin{remark}
Schick and Zadeh~\cite{SZ15} used \emph{coarse} index theory to study obstructions to \emph{uniformly} positive scalar curvature metrics on $M$ in terms of suitable closed submanifolds of arbitrary codimension.
In particular, their approach allows us to deduce Theorem~\ref{T:Theorem A} for complete metrics on $M$ with uniformly positive scalar curvature.
From this point of view, the main novelty of the present paper is that we require strict positivity only in a neighborhood of the submanifold $N$.
\end{remark}


When $M$ is the cylinder $N\times\RR$, Theorem~\ref{T:Theorem A} implies the following consequence.


\begin{corx}\label{C:Corollary B}
Let $N$ be a closed connected spin manifold.
If $N\times\RR$ admits a complete metric whose scalar curvature is nonnegative and strictly positive in a collar neighborhood of $N$, then $\alpha_\CC(N)= 0$.
\end{corx}


\begin{remark}
The ``real" version of this corollary, i.e. with $\alpha_\CC(N)$ replaced by $\alpha_\RR(N)$, would allow us to deduce Conjecture~\ref{C:R-S} for all closed spin manifolds $N$ verifying Conjecture~\ref{C:G-L-R}.
The author plans to treat the case of real $C^\ast$-algebras in a future paper.
\end{remark}


We deduce Theorem~\ref{T:Theorem A} from an abstract index theorem for Callias-type operators twisted with Hilbert $C^\ast$-bundles of finite type.
A Callias-type operator on a complete Riemannian manifold $M$ is an operator of the form $P=D+\Phi$, where $D$ is a Dirac operator and $\Phi$ is a potential, such that $P^2$ is an operator of Schr\"{o}dinger-type with potential uniformly positive at infinity.
This means that $P^2=D^2+\Pi$, where $\Pi$ is a bundle map uniformly positive outside of a compact set.
This condition implies that the spectrum of $P$ is discrete near zero so that $P$ is Fredholm.

The study of such operators was initiated by Callias, \cite{Cal78}, and further developed by many authors, cf., for example, \cite{BottSeeley78}, \cite{BM92}, \cite{Ang93}, \cite{Bun95}. Several generalizations and new applications of the Callias-type index theorem were obtained recently, cf. \cite{Kot11}, \cite{CN14}, \cite{Wim14}, \cite{Kot15}, \cite{BS16}.
For this paper, the relevant property of Callias-type operators is that the computation of their index can be reduced to the computation of the index of a Dirac-type operator on a suitable codimension one closed submanifold.
In particular, the Callias-type index theorem, \cite{Ang93}, \cite{Bun95}, states that the index of $P$ is equal to the index of a certain Dirac operator $D_N$ induced by the restriction of $P$ to a suitable closed hypersurface $N$.

In this paper we suppose that $A$ is a complex \emph{unital} $C^\ast$-algebra, $V$ is a Hilbert $A$-bundle of finite type over $M$ and $P_V$ is the operator obtained by twisting $P$ with the bundle $V$.
We extend to this setting the theory of Callias-type operators.


\begin{thmx}[Callias-type operators in $C^\ast$-algebras]\label{T:Theorem B}
We have:
\begin{enumerate}[label=(\alph*)]
\item \emph{(Invertibility at infinity)} The operator $P_V^2$ is invertible at infinity. Therefore, $P_V$ has a well-defined index class $\ind_A  P_V\in K_0(A)$.
\item \emph{(Callias-type theorem)} 
Suppose $M$ is odd-dimensional and orientable and assume there is a partition $M=M_-\cup_N M_+$, where $N=M_-\cap M_+$ is a closed codimension one submanifold of $M$ and $M_-$ is a compact submanifold, whose interior contains an essential support of $\Phi$ (see Definition~\ref{D:Callias}).
If $A$ is separable, then $\ind_A  P_V =\ind_A  D_{N,i^\ast{V}}$ in $K_0(A)$, where $i:N\hookrightarrow M$ is the inclusion map and $D_{N,i^\ast{V}}$ is the Dirac operator $D_N$, induced by $P$ on $N$, twisted with the pull-back bundle $i^\ast V$.
\end{enumerate}
\end{thmx}


\begin{remark}
For the notion of invertibility at infinity and the definition of the class $\ind_AP_V$, see Subsection~\ref{SS:Bunke's approach}.
\end{remark}


\begin{remark}\label{R:vN1}
Suppose $A$ is a von Neumann algebra endowed with a finite trace $\tau$.
In~\cite{BC16}, it is shown that it is possible to use the trace $\tau$ to define a numerical index $\ind_\tau P_V\in\RR$ and in~\cite{BC} a Callias-type theorem has been proved for such index.
This result cannot be deduced from part (b) of Theorem~\ref{T:Theorem B} since, in general, von Neumann algebras are not separable.
\end{remark}


The possibility of extending the analysis of Callias-type operators to Hilbert $C^\ast$-bundles was left as an open question by Bunke in \cite{Bun95}.
For the invertibility at infinity, the main problem is to show the invertibility of $P_V^2+1$.
To solve this issue, we make use of a recent result of Hanke, Pape and Schick~\cite{HPS15} that guarantees that the operator $P_V^2+1$ has dense range.
In order to prove invertibility, we also show that the operator $P_V^2+1$ has a unique self-adjoint extension.

When $A=\CC$, the proof of the Callias-type theorem used in both~\cite{Ang93} and~\cite{Bun95} consists of two steps.
In the first step, a ``cut-and-paste" argument is used to reduce the initial problem to a model problem on the cylinder $N\times\RR$.
In the second step, a ``separation of variables" argument is used to show that the kernel (resp. cokernel) of the model operator on the cylinder is isomorphic to the kernel (resp. cokernel) of the operator on the base space.
By using the $K$-theoretic relative index theorem of Bunke~\cite{Bun95}, the ``cut-and-paste" argument can be adapted to the case when $A$ is an arbitrary $C^\ast$-algebra.

A second problem arises in the computations on the cylinder.
In the case of arbitrary $C^\ast$-algebras, in order to define the index classes of the operators on the cylinder and on the base space we need to perturb these operators.
This fact doesn't allow us to separate the variables.
We change our point of view here.
We formulate the index classes of the operators on the cylinder and on the base space in a $KK$-theoretical setting and make heavy use of the properties of the Kasparov intersection product to reduce the twisted case to the untwisted one.
In order to do these computations, we use the notion of unbounded connection developed by Kucerovsky in~\cite{Kuc97}.
To this end, we show that, under suitable growth conditions of the endomorphism $\Phi$, the operator $P_V$ defines an unbounded Kasparov cycle.
In order to use the Kasparov product, we need to assume in part (b) that $A$ is separable (see Remark~\ref{R:algebra is separable}).

The paper is organized as follows.
In Section~\ref{S:main results}, we formulate the main results of the paper.
In Section~\ref{S:invertibility at infinity}, we prove that $P_V^2$ is invertible at infinity and use this fact to define the class $\ind_AP_V\in K_0(A)$ (this corresponds to point (a) of Theorem~\ref{T:Theorem B}).
In Section~\ref{S:properties of the index}, we study some properties of this class.
Sections~\ref{S:reduction}, ~\ref{S:growing potentials}, and~\ref{S:analysis on the cylinder} are devoted to the proof of the Callias-type theorem (part (b) of Theorem~\ref{T:Theorem B}).
In particular, Section~\ref{S:reduction} is devoted to the ``cut-and-paste" argument, in Section~\ref{S:growing potentials} we show that the index class of the model operator can be expressed through an unbounded Kasparov module, and in Section~\ref{S:analysis on the cylinder} we present the $KK$-theoretical computations on the cylinder.
This concludes the proof of Theorem~\ref{T:Theorem B}.
In Section~\ref{S:vanishing} we prove a vanishing theorem from which we deduce Theorem~\ref{T:Theorem A}.
Finally, in Appendix~\ref{A:shroedinger-type} we prove regularity and self-adjointness of Schr\"odinger-type operators on Hilbert $C^\ast$-bundles of finite-type.

Our analytic results of Section~\ref{S:invertibility at infinity} have intersection with~\cite[Section~3.4]{Ebe16}.


\subsection*{Acknowledgment} 
This paper would not have appeared without Maxim Braverman's help. 
I am very grateful to him for countless remarks and suggestions.
The author wishes also to thank Jens Kaad, Paolo Piazza, Thomas Schick and Yanli Song for very useful and enlightening discussions.


\section{Main results}\label{S:main results}
In this section we formulate the main results of the paper.


\subsection{Twisted Dirac-type operators}\label{SS:twisted Dirac}
Let $M$ be a complete Riemannian manifold and let $S$ be a complex Dirac bundle over $M$.
This means that $S$ is a complex vector bundle endowed with a Clifford action $c:T^*M\to \End(S)$ of the cotangent bundle and a metric connection $\nabla^S$ compatible with the inner product of the fibers and satisfying the Leibniz rule (see \cite[Definition~5.2.]{LM89}).
The Dirac operator associated to this bundle is the formally self-adjoint operator $\slashed{D}\in \Diff^1(M;S)$ given by the composition
\[
	C^\infty_c(M;S)\xrightarrow{\,\nabla^S\,}
	C^\infty_c(M;T^\ast M\tensor S)\xrightarrow{\ c\ }
	C^\infty_c(M;S)\,.
\]
Fix a self-adjoint potential $\Psi\in C^\infty(M;\End(S))$ and consider the Dirac-type operator
\begin{equation}
	B\ :=\slashed{D}\,+\,\Psi\,.
\end{equation}
Notice that the operator $B$ is formally self-adjoint.


Let $A$ be a complex unital $C^\ast$-algebra and let $V$ be a Hilbert $A$-bundle of finite type over $M$. 
In particular, this means that the fibers of $V$  are finitely generated projective Hilbert $A$-modules.
We also suppose that $V$ is endowed with a metric connection $\nabla^V$ preserving the $A$-valued inner product of the fibers.
The tensor product $S\tensor V$ is a Hilbert $A$-bundle of finite type.
The Dirac operator $\slashed{D}$ twisted with the bundle $V$ is the operator $\slashed{D}_V\in \Diff^1(M;S\tensor V)$ defined through the composition
 \begin{equation}\label{eq:twisted Dirac}
	C^\infty_c(M;S\tensor V)\xrightarrow{\,\nabla^S\tensor 1+1\tensor \nabla^V\,}
	C^\infty_c(M;T^\ast M\tensor S\tensor V)\xrightarrow{\ c\,\tensor 1\ }
	C^\infty_c(M;S\tensor V)\,.
\end{equation}
We also extend the potential $\Psi$ to a section $\Psi_V\in C^\infty(M;\End_A(S\tensor V))$ by setting
\[
	\Psi_V\ :=\ \Psi\tensor 1\,.
\]

\begin{definition}\label{D:twisted Dirac-type}
The Dirac-type operator $B=\slashed{D}+\Psi$ \emph{twisted with the bundle} $V$ is the operator $B_V:=\slashed{D}_V+\Psi_V$.
\end{definition}


\subsection{Sobolev spaces}\label{SS:Sobolev}
Fix a nonnegative integer $l$. 
We use the operator $B_V$ to define the $A$-valued inner product
\[
	\left< u,v\right>_l\ :=\ \sum_{k=0}^l\int_M \left< \big(B_V^ku\big)(x),\big(B_V^kv\big)(x)\right>_x\,d\mu(x)\,,
	\qquad\qquad u,\,v\in C^\infty_c(M;S\tensor V)\,,
\]
where $d\mu(x)$ is the smooth measure induced by the Riemannian metric on $M$ and $\left<\cdot ,\cdot\right>_x$ denotes the $A$-valued inner product of the fiber $S_x\tensor V_x$.
Endowed with this inner product, $C^\infty_c(M;S\tensor V)$ has a pre-Hilbert $A$-module structure.
We denote by $H^l$ the Hilbert $A$-module obtained as the completion of $C^\infty_c(M;S\tensor V)$ with respect to the norm
\[
	\|u\|_l\ :=\ \sqrt{|\left< u,u\right>_l|_A}\,\,,\qquad\qquad\qquad u\in C^\infty_c(M;S\tensor V)\,,
\]
where $|\cdot|_A$ denotes the norm of the $C^\ast$-algebra $A$.

We denote by $\mathcal{L}_A(H^i,H^j)$ the space of bounded adjointable $A$-linear operators from $H^i$ to $H^j$ (for basic notions on Hilbert $C^\ast$-modules and bounded adjointable operators, cf.~\cite[Chapter~1]{Lan95}).
Given an operator $T\in\mathcal{L}_A(H^i,H^j)$, we denote by $\|T\|_{\mathcal{B}(H^i,H^j)}$ the operator norm of $T$ as a bounded operator $H^i\rightarrow H^j$.
Finally, we set $\mathcal{L}_A(H^i):=\mathcal{L}_A(H^i,H^i)$ and $\|\cdot\|_{\mathcal{B}(H^i)}:=\|\cdot\|_{\mathcal{B}(H^i,H^i)}$.


\subsection{A Schr\"odinger-type operator}\label{SS:A schrodinger-type}
We regard $B_V=\slashed{D}_V+\Psi_V$ as an unbounded operator on $H^0$ with initial domain $C^\infty_c(M;S\tensor V)$.
By~\cite[Theorem~2.3]{HPS15}, $B_V$ has a unique extension to a regular self-adjoint operator on $H^0$ (for the notion of regularity and more in general for basic notions of unbounded operators on Hilbert $A$-modules, we refer the reader to~\cite[Chapter~9]{Lan95}).

Fix a self-adjoint potential $\Pi\in \Gamma(M;S)$ and consider the Schr\"odinger-type operator 
\begin{equation}\label{eq:B_V^2+mu^2}
	G_V\ :=\ B_V^2 \ + \ \Pi_V\,,
\end{equation}
where $\Pi_V:=\Pi\tensor\id_V$.
We view $G_V$ as an $A$-linear unbounded operator on $H^0$ with initial domain $C^\infty_c(M;S\tensor V)$.


\begin{theorem}\label{T:H_mu essentially self-adjoint}
Suppose that $\Pi$ is uniformly bounded from below.
Then the minimal closure of $G_V$ is a regular, self-adjoint operator on $H^0$.
It is the unique self-adjoint extension of $G_V$.
\end{theorem}


\begin{remark}
This theorem is a direct consequence of Theorem~\ref{T:G essentially self-adjoint} proved in Appendix~\ref{A:shroedinger-type}.
\end{remark}


\subsection{$A$-index of twisted Dirac-type operators: Bunke's approach}\label{SS:Bunke's approach}
Let $q:M\rightarrow\RR$ be a smooth function which is constant outside of a compact set.
It is a classical fact (see \cite{FM80}) that, for every nonnegative integer $j$, the operator $B_V^k+q$ extends to a bounded adjointable operator
\begin{equation*}\label{eq:P_v^k+f}
	B_V^k+q\colon H^{j+k}\longrightarrow H^j\,.
\end{equation*}


\begin{definition}\label{D:invertible at infinity}
We say that the operator $B_V^2$ is \emph{invertible at infinity} if there exists a compactly supported smooth function $f\colon M\rightarrow [0,\infty)$ such that the operator $B_V^2+f$ is invertible and $\big(B_V^2+f\big)^{-1}\in\mathcal{L}_A(H^0,H^2)$.
\end{definition}


In this case, Bunke associated to the operator $B_V$ a $K$-theoretical index class.
The construction of this class makes use of Kasparov's $KK$-theory.
For the notion of bounded Kasparov module and $KK$-group, we refer to~\cite{Bla98}. 


\begin{theorem}[Bunke,~\cite{Bun95}]\label{T:Bunke}
Suppose $B_V^2$ is invertible at infinity and let $f\in C^\infty_c(M)$ be as in Definition~\ref{D:invertible at infinity}.
Then the triple 
\begin{equation}\label{eq:KK-index of P_V}
	\left(H^0,1,B_V\left( B_V^2+f\right)^{-1/2}\right)
\end{equation}
is a \emph{bounded Kasparov module} for the pair of algebras $(\CC,A)$.
Moreover, the class in $KK(\CC,A)= K_0(A)$ defined by the triple~\eqref{eq:KK-index of P_V} is independent of the choice of the function $f$.
In this case, the $A$-index of $B_V$ is the $K$-theoretical class
\begin{equation}\label{eq:ind_AP_V}
	\ind_AB_V\ :=\ \left[H^0,1,B_V\left( B_V^2+f\right)^{-1/2}\right]\in KK(\CC,A)\ =\ K_0(A)\,,
\end{equation}
where $f$ is as in Definition~\ref{D:invertible at infinity}.
\end{theorem}


\begin{remark}\label{L:integral representation}
The operator $B_V\left( B_V^2+f\right)^{-1/2}$ used in~\eqref{eq:KK-index of P_V} is defined as follows. 
For $w\in H^1$, set
\begin{equation}\label{eq:integral representation}
	 B_V\left( B_V^2+f\right)^{-1/2}w\ :=\ \frac{2}{\pi}\int_0^\infty B_V\left( B_V^2+f+\lambda^2\right)^{-1}w\, d\lambda\,.
\end{equation}
In~\cite[Lemma~1.8]{Bun95} it is shown that the integral on the right hand side is norm-convergent and defines an operator in $\mathcal{L}_A(H^0)$.
Moreover, by~\cite[Proposition~1.13]{Bun95} and the construction of~\cite[page~244]{Bun95}, the triple~\eqref{eq:KK-index of P_V} defines a class in $KK(\CC,A)$ that is independent of the choice of $f$.
\end{remark}


We now define a particular class of Dirac-type operators and use Bunke's approach to define the $A$-index of the operators in this class.


\subsection{Twisted Callias-type operators}\label{SS:twisted Callias}
Let $M$ be an odd-dimensional complete Riemannian manifold and let $\Sigma$ be an \emph{ungraded} complex Dirac bundle over $M$.
Let $D\in\Diff^1(M;\Sigma)$ be a formally self-adjoint  Dirac-type operator (see Subsection~\ref{SS:twisted Dirac}).
Fix a self-adjoint potential $\Phi\in C^\infty(M;\End(\Sigma))$ and consider the operator
\begin{equation*}
	P\ :=\ \left(\begin{array}{cc}0&D- i\Phi\\D+ i\Phi&0\end{array}\right)\,.
\end{equation*}
Notice that $P\in\Diff^1(M;\Sigma\oplus\Sigma)$ is of \emph{Dirac-type} and \emph{formally self-adjoint}.
Let $V$ be as in Subsection~\ref{SS:twisted Dirac} and consider the twisted Dirac-type operator
\begin{equation}\label{eq:twisted Callias}
	P_V\ :=\ \left(\begin{array}{cc}0&D_V- i\Phi_V\\D_V+ i\Phi_V&0\end{array}\right)\,,
\end{equation}
where $\Phi_V:=\Phi\tensor \id_V$.


\begin{definition}\label{D:Callias}
The endomorphism $\Phi$ is said to be \emph{admissible} for the pair $(\Sigma,D)$ if
\begin{enumerate}
\item[(i)] the commutator $[D,\Phi]:=D\Phi-\Phi D$ is an endomorphism of $\Sigma$;
\item[(ii)] there exist a constant $d>0$ and a compact set $K\subset M$ such that 
\begin{equation}\label{eq:ungraded admissible endomorphism}
	\Phi^2(x)\ \geq\  d\ +\ \left\|[D,\Phi](x)\right\|,\qquad x\in M\setminus K\,.
\end{equation}
\end{enumerate}
In this case, we say that $K$ is an \emph{essential support} for $\Phi$ and that the operator $P_V$ defined in~\eqref{eq:twisted Callias} is the \emph{twisted Callias-type operator} associated to the \emph{admissible quadruple} $(\Sigma,D,\Phi,V)$.
If we can choose $K=\emptyset$, we say that $\Phi$ has \emph{empty} essential support.
\end{definition}


In Inequality~\eqref{eq:ungraded admissible endomorphism} we used the following notation, that will be used throughout this paper without specific mention.


\begin{notation}
Let $W\rightarrow M$ be a Hilbert $A$-bundle of finite type. 
For $\Psi_0$, $\Psi_1\in C^\infty(M;\End(W))$ and $x\in M$, we say that $\Psi_0(x)\geq \Psi_1(x)$ if $\left<\Psi_0(x)v,v\right>_x\geq \left<\Psi_1(x)v,v\right>_x$ for all $v\in W_x$, where $\left<\cdot,\cdot\right>_x$ is the inner product of the fiber $W_x$.
\end{notation}


\subsection{$A$-index of Callias-type operators}\label{SS:A-index of P}
In order to define the index of a Callias-type operator $P_V$, we use Bunke's approach (see Subsection~\ref{SS:Bunke's approach}).


\begin{theorem}\label{T:B^2 is invertible at infinity}
Let $M$ be a complete Riemannian manifold and let $P_V$ be the Callias-type operator associated to an admissible quadruple $(\Sigma,D,\Phi,V)$.
Then the operator $P_V^2$ is invertible at infinity.
Therefore, $P_V$ has an index class $\ind_AP_V$ in $K_0(A)$ defined by~\eqref{eq:ind_AP_V}.
\end{theorem}


\begin{remark}
This theorem corresponds to part (a) of Theorem~\ref{T:Theorem B} and will be proved in Section~\ref{S:invertibility at infinity}.
\end{remark}


\subsection{A Callias-type theorem}\label{SS:Callias-type theorem}
Let $(\Sigma,D,\Phi,V)$ be an admissible quadruple on a complete Riemannian manifold $M$ and let $P_V$ be the associated Callias-type operator.
Suppose that there is a partition $M=M_-\cup_N M_+$, where $N=M_-\cap M_+$ is a smooth compact hypersurface and $M_-$ is a compact submanifold, whose interior contains an essential support of $\Phi$.
We want to use our data to construct a twisted Dirac-type operator on $N$ and use this operator to compute the class $\ind_AP_V$.

Let $\Sigma_N$ be the  restriction of $\Sigma$ to $N\subset M$.
Condition (ii) of Definition~\ref{D:Callias} implies that zero is not in the spectrum of $\Phi(x)$ for all $x\in N$.
Therefore, we have the splitting
\begin{equation}\label{eq:decomposition of Sigma_N}
	\Sigma_N=\Sigma_{N+}\oplus \Sigma_{N-}\,,
\end{equation}
where $\Sigma_{N+}$, $\Sigma_{N-}$ are respectively the positive and negative eigenbundles of $\Phi$.

By Condition~(i) of Definition~\ref{D:Callias}, the endomorphism $\Phi$ commutes with the Clifford multiplication. 
Hence, $c(\xi):\Sigma_{N\pm}\to \Sigma_{N\pm}$ for all $\xi\in T^*M\big|_N$. 
It follows that both bundles, $\Sigma_{N+}$ and $\Sigma_{N-}$, inherit  the Clifford action of $T^*M$.
In particular, Clifford multiplication by the unit normal vector field pointing at the direction of $M_+$ defines an endomorphism $\gamma:\Sigma_{N\pm}\to \Sigma_{N\pm}$.
Since $\gamma^2=-1$, the endomorphism $\alpha:=-i \gamma$ induces a grading 
\begin{equation}\label{eq:grading on SigmaN}
	\Sigma_{N\pm}\ = \ \Sigma_{N\pm}^+\oplus \Sigma_{N\pm}^-\,,
\end{equation}
where $\Sigma_{N\pm}^\pm$ is the span of the eigenvectors of $\alpha$ with eigenvalues $\pm1$.

We use the Riemannian metric on $M$ to identify $T^*N$ with a subbundle of $T^*M$. Then the Clifford action of $T^*N$ on $\Sigma_{N\pm}$ is graded with respect to this grading, i.e. $c(\xi):\Sigma_{N\pm}^\pm\to \Sigma_{N\pm}^\mp$ for all $\xi\in T^*N$. 
Let $\nabla^{\Sigma_N}$ be the connection on $\Sigma_N$ obtained by restricting  the connection on $\Sigma$. It does not, in general, preserve decomposition~\eqref{eq:decomposition of Sigma_N}. 
We define a connection $\nabla^{\Sigma_{N\pm}}$ on the bundle $\Sigma_{N\pm}$ by setting 
\begin{equation*}
	\nabla^{\Sigma_{N\pm}} s^\pm\ :=\ 
	\pr_{\Sigma_{N\pm}}\left(\nabla^{\Sigma_N} s^\pm\right)\,,
	\qquad\qquad s^\pm\in C^\infty(N;\Sigma_{N\pm})\,,
\end{equation*}
where $\pr_{\Sigma_{N\pm}}$ is the projection onto the bundle $T^\ast N\tensor\Sigma_{N\pm}$. 
By \cite[Lemma~2.7]{Ang90} (see also \cite[Section~5.1]{BC}), $\Sigma_{N+}$ and $\Sigma_{N-}$ carry a $\ZZ_2$-graded Dirac bundle structure.

We denote by  $ D_{N+}$, $D_{N-}$ the Dirac operators on $N$ associated respectively with the bundles $\Sigma_{N+}$ and $\Sigma_{N-}$. 
Notice that the operators $D_{N\pm}$ are odd  with respect to the grading \eqref{eq:grading on SigmaN}, i.e. they have the form 
\[
	D_{N\pm}\ = \ \begin{pmatrix}
	0&D_{N\pm}^-\vspace{0.2cm}\\D_{N\pm}^+&0
	\end{pmatrix}\,,
\]
where $D_{N\pm}^+$ (respectively $D_{N\pm}^-$) is the restriction of $D_{N\pm}$ to $\Sigma_{N\pm}^+$ (respectively $\Sigma_{N\pm}^-$). 

Let $V_N$ be the restriction of $V$ to $N$. 
It is a Hilbert $A$-bundle of finite type endowed with a connection $\nabla^V_N$ obtained by pulling back the connection $\nabla^V$.
Consider the operator $D_{N+,V_N}$ obtained by twisting the Dirac operator $D_{N+}$ with the bundle $V_N$.
In the classical paper \cite{FM80}, Mi\v{s}\v{c}enko and Fomenko showed that the operator $D_{N+,V_N}$ is $A$-Fredholm and it has a well-defined index class $\ind_A D_{N+,V_N}\in K_0(A)$ (see also \cite[Section~5]{Sch05}).

The next theorem is the main result of this paper. 
In the case when $A=\CC$ it was proved in \cite[Theorem~1.5]{Ang93} and \cite[Theorem~2.9]{Bun95}.
When $A$ is a von Neumann algebra with a finite trace, the result has been recently proved in~\cite{BC}.


\begin{theorem}[Callias-type theorem in $C^\ast$-algebras]\label{T:Callias-type theorem}
Suppose that the $C^\ast$-algebra $A$ is separable. Then the classes $\ind_AP_V$ and $\ind_A D_{N+,V_N}$ coincide.
\end{theorem}


\begin{remark}
This theorem corresponds to part (b) of Theorem~\ref{T:Theorem B} and will be proved in Sections~\ref{S:reduction}, ~\ref{S:growing potentials}, and~\ref{S:analysis on the cylinder}.
\end{remark}


\begin{remark}\label{R:A separable}
The proof of this theorem consists of two steps.
In the first step, we reduce the computation of $\ind_AP_V$ to the computation of the $A$-index of a \emph{model operator} ${\bf M}_{V_N}$, which is a Callias-type operator on the cylinder $N\times\RR$.
The second step consists of solving a problem on the cylinder, i.e. proving Theorem~\ref{T:Callias-type theorem} for the operator ${\bf M}_{V_N}$.
In order to do the computations on the cylinder, we reformulate the problem in a $KK$-theoretical setting and make use of the properties of the intersection product.
\end{remark}


\begin{remark}\label{R:unbounded modules}
In Subsection~\ref{SS:model operator}, we define the operator ${\bf M}_{V_N}$ by using a potential growing to infinity at infinity but with uniformly bounded first derivatives.
Because of this choice, the operator ${\bf M}_{V_N}$ has compact resolvent (see~Section~\ref{S:growing potentials}).
This allows us to work with unbounded Kasparov modules in doing our $KK$-theoretical computations on the cylinder.
\end{remark}


\begin{remark}\label{R:interpretation of the index}
Theorem~\ref{T:Callias-type theorem} reduces the computation of the index class of an elliptic operator on a noncompact manifold to the computation of the index class of an operator on a suitable closed submanifold.
Therefore, we can use the rich theory of elliptic operators on closed manifolds to interpret such class.
The application of this theorem that we present in this paper is based on this fact.
\end{remark}


\subsection{Codimension one obstructions to PSC on noncompact manifolds}\label{SS:codimension 1 obstructions}
Let $(M,g)$ be a complete odd-dimensional oriented Riemannian spin manifold.
Denote by $\kappa$ the scalar curvature of $g$.
Suppose there is a partition $M=M_-\cup_N M_+$, where $N=M_-\cap M_+$ is a smooth closed hypersurface.
Notice that the normal bundle of $N$ is trivial so that, by \cite[Proposition~2.15]{LM89}, $N$ is endowed with a spin structure and the operator $\slashed{D}_{N,V_N}$ is well-defined.


\begin{theorem}\label{T:codimension one obstructions}
Let $V\rightarrow M$ be a flat Hilbert $A$-bundle.
Suppose that the scalar curvature $\kappa$ is nonnegative and there exists a tubular neighborhood $U$ of $N$ such that $\kappa$ is strictly positive on $U$.
Then the class $\ind_A \slashed{D}_{N,V_N}\in K_0(A)$ vanishes.
\end{theorem}


\begin{remark}
In Section~\ref{S:vanishing}, we use Theorem~\ref{T:codimension one obstructions} to prove Theorem~\ref{T:Theorem A}.
\end{remark}

\section{Invertibility at infinity of Callias-type operators}\label{S:invertibility at infinity}
This section is devoted to the proof of Theorem~\ref{T:B^2 is invertible at infinity}.
We let $M$, $S$, $V$, $B_V$ denote the same objects of Subsection~\ref{SS:twisted Dirac}.
We deduce Theorem~\ref{T:B^2 is invertible at infinity} from the following theorem.


\begin{theorem}\label{T:B_V^2+f invertible}
Suppose there exist a function $f\in C^\infty_c(M)$ and a constant $c>0$ such that 
\begin{equation}\label{B_V^2>c>0}
	\left<\left(B_V^2\ +\ f\right)s,s\right>_0\ \geq \ c\,\left<s,s\right>_0\,,\qquad\qquad s\in H^2\,.
\end{equation}
Then $B_V^2 +f$ is invertible with $\big(B_V^2 +f\big)^{-1}\in\mathcal{L}_A(H^0)\cap\mathcal{L}_A(H^0,H^2)$.
\end{theorem}


Before proving this theorem, let us deduce some consequences.


\begin{corollary}\label{C:condition for ind_AB_V=0}
Suppose there is a constant $c>0$ such that
\begin{equation}\label{B_V^2 >0}
	\left<B_V^2s,s\right>_0\ \geq \ c\,\left<s,s\right>_0\,,\qquad\qquad s\in H^2\,.
\end{equation}
Then the class $\ind_AB_V$ vanishes.
\end{corollary} 


\begin{proof}
By Theorem~\ref{T:B_V^2+f invertible}, Condition~\ref{B_V^2 >0} implies that $B_V^2$ is invertible, from which the thesis follows.
\end{proof}


\subsection{Proof of Theorem~\ref{T:B^2 is invertible at infinity}}\label{Pf:B^2 invertible at infinity}
Let $P_V$ be the Callias-type operator associated to an admissible quadruple $(\Sigma,D,\Phi,V)$ over a complete Riemannian manifold $M$.
We have
\begin{equation}\label{eq:P_V^2+f}
	P_V^2\ =\ \left(\begin{array}{cc}D_V^2+\Phi_V^2+i\big([D_V,\Phi_V]\big)&0\vspace{0.2cm}\\
	0&D_V^2+\Phi_V^2-i\big([D_V,\Phi_V]\big)\end{array}\right)\,.
\end{equation}
By Conditions~(i) and~(ii) of Definition~\ref{D:Callias}, the commutator $\big[D,\Phi\big]$ is in $C^\infty(M;\End(\Sigma))$ and we can choose a compactly supported smooth function $f:M\rightarrow [0,\infty)$ such that
\begin{equation}\label{eq:B_V^2 geq c}
	\Phi^2(x)\,+f(x)\,-\,\big\|\big[D,\Phi\big](x)\big\|\ \geq\ c\,,\qquad\qquad x\in M\,,
\end{equation}
for some constant $c>0$.
In~\cite{BC} (see the proof of Lemma~8.2) it is shown that 
\begin{equation}\label{eq:[D_V,phi_V]=[D,phi]}
	\big[D_V,\Phi_V\big]\ =\ \big[D,\Phi\big]\tensor\id_V\,.
\end{equation}
From~\eqref{eq:P_V^2+f},~\eqref{eq:B_V^2 geq c}, and~\eqref{eq:[D_V,phi_V]=[D,phi]} we deduce that Condition~\eqref{B_V^2>c>0} is satisfied.
Now the thesis follows from Theorem~\ref{T:B_V^2+f invertible}
\hfill$\square$


\begin{corollary}\label{C:vanishing sufficient condition}
Let $P_V$ be a twisted Callias-type operator associated to an admissible quadruple $\big(\Sigma,D,\Phi,V\big)$ over a complete Riemannian manifold $M$. 
If $\Phi$ has empty essential support (see Definition~\ref{D:Callias}), then the class $\ind_A P_V$ vanishes.
\end{corollary}


\begin{proof}
If $\Phi$ has empty essential support, then Inequality~\eqref{eq:B_V^2 geq c} holds with $f=0$.
Now the thesis follows from Corollary~\ref{C:condition for ind_AB_V=0}.
\end{proof}


The remaining part of this section is devoted to proving Theorem~\ref{T:B_V^2+f invertible}.


\subsection{Invertibility of $B_V^2+\mu^2$}\label{SS:B_V^2+lambda^2}
We fix a constant $\mu\in\RR\setminus\{0\}$ and study the bounded adjointable operator
\begin{equation}\label{eq:H_mu bounded}
	B_V^2+\mu^2\colon H^2\longrightarrow H^0\,.
\end{equation}
To this end, we need some information about the operator $B_V$.
We view $B_V$ as an unbounded operator on $H^0$ with initial domain $C^\infty_c(M;S\tensor V)$.
We make use of the following result.


\begin{theorem}[Hanke-Pape-Schick,~\cite{HPS15}]\label{T:H-P-S}
The minimal closure $\overline{B_V}$ of the operator $B_V$ is a regular self-adjoint operator.
It is the unique self-adjoint extension of $B_V$.
\end{theorem}


\begin{remark}
This theorem has been recently generalized by J.~Ebert to a larger class of first order differential operators acting on Hilbert $A$-bundles of finite type (see~\cite[Theorem~2.14]{Ebe16}).
\end{remark}


\begin{proposition}\label{P:invertibility of B_V^2+lambda^2}
The operator $B_V^2+\mu^2$ is invertible. 
Moreover, $\left(B_V^2+\mu^2\right)^{-1}\in \mathcal{L}_A(H^0)\cap \mathcal{L}_A(H^0, H^2)$.
\end{proposition}


\begin{proof}
Consider the minimal closure 
\begin{equation}\label{eq:H_mu closure}
	\overline{B_V^2+\mu^2}\colon\Dom\big(\overline{B_V^2+\mu^2}\big)\longrightarrow H^0
\end{equation}
of the unbounded operator $B_V^2+\mu^2\colon C^\infty_c(M;S\tensor V)\longrightarrow H^0$.
By Theorem~\ref{T:H_mu essentially self-adjoint}, the operator $\overline{B_V^2+\mu^2}$ is self-adjoint and regular.
Moreover, it is strictly positive, since $\left<\overline{B_V^2+\mu^2}u,u\right>\geq\mu^2\left<u,u\right>$ for $u\in \Dom\big(\overline{B_V^2+\mu^2}\big)$.
By~\cite[Theorem~2.21]{Ebe16}, $\overline{B_V^2+\mu^2}$ is invertible and $\left(\overline{B_V^2+\mu^2}\right)^{-1}$ is a positive element of $\mathcal{L}_A(H^0)$.

We now compare operators~\eqref{eq:H_mu bounded} and~\eqref{eq:H_mu closure}.
Recall that the domain of $\overline{B_V^2+\mu^2}$ is the closure of $C^\infty_c(M;S\tensor V)$ in $H^0$ with respect to the graph norm
\[
	\|u\|_\Gamma\,:=\,\sqrt{\|\left(B_V^2+\mu^2\right)u\|^2+\|u\|^2}\,,\qquad u\in C^\infty_c(M;S\tensor V)\,.
\]
Since the norms $\|\cdot\|_2$ and $\|\cdot\|_\Gamma$ are equivalent on $C^\infty_c(M;S\tensor V)$, then $\Dom(\overline{B_V^2+\mu^2})$ and $H^2$ coincide as sets and are isomorphic as Banach spaces.
It follows that the operators~\eqref{eq:H_mu bounded} and~\eqref{eq:H_mu closure} coincide as maps from $H^2=\Dom\big(\overline{B_V^2+\mu^2}\big)$ to $H^0$.
Therefore, by the first part of the proof, the operator~\eqref{eq:H_mu bounded} is invertible and its inverse is a positive element of $\mathcal{L}_A(H^0)$.
It remains to show that $\left(B_V^2+\mu^2\right)^{-1}$ is in $\mathcal{L}_A(H^2,H^0)$.

By the open mapping theorem, the $A$-linear operator $\left(B_V^2+\mu^2\right)^{-1}\colon H^0\longrightarrow H^2$ is bounded.
To show adjointability, we make use of the bounded adjointable operator $B_V\colon H^1\longrightarrow H^0$.
Since $H^1$ coincides with $\Dom(\overline{B_V})$, by Theorem~\ref{T:H-P-S} we have
\begin{equation}\label{eq:P_V is SA}
	\left<B_Vs,s\right>_0\ =\ \left<s,B_Vs\right>_0\,,\qquad\qquad s\in H^1\,.
\end{equation}
Moreover, since $\left(B_V^2+\mu^2\right)^{-1}$ is positive in $\mathcal{L}_A(H^0)$, by \cite[proof of Lemma~4.1]{Lan95} it is also self-adjoint, i.e.
\begin{equation}\label{eq:Q is s.a.}
	\left< \left(B_V^2+\mu^2\right)^{-1}u,v\right>_0\ =\  \left< u,\left(B_V^2+\mu^2\right)^{-1}v\right>_0\,,
	\qquad\qquad u,\,v\in H^0\,.
\end{equation}
We are now ready to show that $\left(B_V^2+\mu^2\right)^{-1}$ is adjointable as operator $H^0\rightarrow H^2$.
Fix $u\in H^0$, $w\in H^2$ and set $Q:=\left(B_V^2+\mu^2\right)^{-1}$.
Observe that $B_Vw\in H^1$ and $B_V^2w\in H^0$.
Using~\eqref{eq:P_V is SA} and~\eqref{eq:Q is s.a.}, we obtain
\[
\begin{aligned}
	\left< Qu,w\right>_2\ =\ &
	\left< B_V^2Qu,B_V^2w\right>_0\,+\,\left< B_VQu,B_Vw\right>_0\,+\,\left< Qu,w\right>_0\vspace{0.2cm}\\
	=\ &\left< \left(B_V^2\,+\,\mu^2\right)Qu,B_V^2w\right>_0\,+\,\left(1-\mu^2\right)\left< Qu,B_V^2w\right>_0\,+\,
	\left< u,Qw\right>_0\vspace{0.2cm}\\
	=\ &\left< u,B_V^2w\right>_0\,+\,\left(1-\mu^2\right)\left< u,QB_V^2w\right>_0\,+\,
	\left< u,Qw\right>_0\vspace{0.2cm}\\
	=\ &\left< u,\left\{ B_V^2\,+\,\left(1-\mu^2\right)QB_V^2\,+\,Q\right\}w\right>_0\,,
\end{aligned}
\]
from which it follows that $Q=\left(B_V^2+\mu^2\right)^{-1}$ is adjointable also as operator $H^0\rightarrow H^2$.
\end{proof}


\subsection{Invertibility at infinity of $B_V^2$}\label{SS:invertibilty of B_V^2+f}
In this subsection we conclude the proof of Theorem~\ref{T:B_V^2+f invertible}.


\begin{lemma}\label{L:invertibility of B_V^2+lambda^2+f}
Let $\mu$ be a nonzero real number and let $f\in C^\infty_c(M)$ be such that $\mu^2>\|f\|_\infty$.
Then the operator $B_V^2+\mu^2+f$ is invertible. Moreover, $\left(B_V^2+\mu^2+f\right)^{-1}\in \mathcal{L}_A(H^0)\cap \mathcal{L}_A(H^0, H^2)$.
\end{lemma}


\begin{proof}
By Proposition~\ref{P:invertibility of B_V^2+lambda^2}, $B_V^2+\mu^2$ is invertible and the inverse is in both $\mathcal{L}_A(H^0)$ and $\mathcal{L}_A(H^0,H^2)$.
In particular, we can write
\begin{equation}\label{eq:decomposition of B_V^2+lambda^2+f}
	B_V^2+\mu^2+f\ =\ \left\{ \Id+f\left( B_V^2+\mu^2\right)^{-1}\right\} \left( B_V^2+\mu^2\right)\,.
\end{equation}
Since $\mu^2>\|f\|_\infty$, we have
\[
	\left< \left(B_V^2+\mu^2\right)s,s\right>_0\ \geq\ \mu^2\left< s,s\right>_0\ >\ \|f\|_\infty\left< s,s\right>_0\,,
	\qquad\qquad s\in H^2\,.
\]
By the Cauchy-Schwarz inequality for Hilbert $A$-modules, we obtain
\begin{equation}\label{eq:norm (B_V^2+lambda^2)^-1}
	\big\|\big( B_V^2+\mu^2\big)^{-1}\big\|_{\mathcal{B}(H^0)}\ <\ \|f\|_\infty\,.
\end{equation}
Since $\left(B_V^2+\mu^2\right)^{-1}\in \mathcal{L}_A(H^0)$ and $f\in L^\infty(M)$, then also $f\left(B_V^2+\mu^2\right)^{-1}\in \mathcal{L}_A(H^0)$ and by \eqref{eq:norm (B_V^2+lambda^2)^-1} we deduce
\[
	\left\|\,f\,\big( B_V^2+\mu^2\big)^{-1}\right\|_{\mathcal{B}(H^0)}\ <\ 1\,.
\]
Therefore, the operator $\Id+f\left(B_V^2+\mu^2\right)^{-1}$ is invertible with bounded inverse given by the Neumann series
\begin{equation*}\label{eq:neumann series 2}
	\sum_{k=0}^\infty \,(-1)^k\,\left\{f\left(B_V^2+\mu^2\right)^{-1}\right\}^k\,.
\end{equation*}
Since this series converges in norm and each summand is adjointable, it defines an element in $\mathcal{L}_A(H^0)$.
Since $\big(B_V^2+\mu^2\big)^{-1}$ is in $\mathcal{L}_A(H^0)\cap \mathcal{L}_A(H^0,H^2)$, from \eqref{eq:decomposition of B_V^2+lambda^2+f} we finally deduce that $B_V^2+\mu^2+f$ is invertible with inverse in $\mathcal{L}_A(H^0)\cap \mathcal{L}_A(H^0,H^2)$ given by the norm convergent series
\[
	\left(B_V^2+\mu^2\right)^{-1} \sum_{k=0}^\infty \,(-1)^k\,\left\{f\left(B_V^2+\mu^2\right)^{-1}\right\}^k .
\]
\end{proof}


\subsection{Proof of Theorem~\ref{T:B_V^2+f invertible}}
Choose a function $f\in C^\infty_c(M)$ and a constant $c>0$ such that Condition~\eqref{B_V^2>c>0} is satisfied.
We want to show that $B_V^2+f$ is invertible and $(B_V^2+f)^{-1}\in\mathcal{L}_A(H^0)\cap \mathcal{L}_A(H^0, H^2)$.

Pick $\mu\neq 0$ such that $\mu^2>\|f\|_\infty$. 
By Lemma~\ref{L:invertibility of B_V^2+lambda^2+f}, the operator $B_V^2+\mu^2+f$ is invertible and $\left(B_V^2+\mu^2+f\right)^{-1}\in \mathcal{L}_A(H^0)\cap \mathcal{L}_A(H^0, H^2)$. In particular, we can write
\begin{equation}\label{eq:neumann series of D^2+f}
	B_V^2+f\ =\ \left[\Id\,-\,\mu^2\left( B_V^2+\mu^2+f\right)^{-1}\right] \left(B_V^2+\mu^2+f\right)\,.
\end{equation}
By Condition~\eqref{B_V^2>c>0}, we have
\[
	\left<\left(B_V^2+\mu^2+f\right)s,s\right>_0\ \geq\ \left(\mu^2+c\right)\,\left<s,s\right>_0\,.
\]
By the Cauchy-Schwarz inequality for Hilbert $A$-modules, we obtain
\begin{equation}\label{eq:norm (B_V^2+f+lambda^2)^-1}
	\left\|\mu^2\left(B_V^2+\mu^2+f\right)^{-1}\right\|_{\mathcal{B}(H^0)}\ \leq\  \frac{\mu^2}{\mu^2+c} \ <\ 1\,.
\end{equation}
Using a Neumann series in the same way as in the proof of Lemma~\ref{L:invertibility of B_V^2+lambda^2+f}, from~\eqref{eq:neumann series of D^2+f} and~\eqref{eq:norm (B_V^2+f+lambda^2)^-1} we deduce the thesis.
\hfill$\square$

\section{Some properties of the index of twisted Callias-type operators}\label{S:properties of the index}
We establish some properties of the index class of twisted Callias-type operators.
These properties will be used to do the deformations and ``cut-and-paste" constructions of Section~\ref{S:reduction}.

Let $M$, $S$, $V$ be as in Subsection~\ref{SS:twisted Dirac} and let $B_V\in\Diff^1(M;S\tensor V)$ be a twisted formally self-adjoint Dirac-type operator.
Let $\Psi\in C^\infty(M;\End(S))$ be a self-adjoint potential and consider the operator
\begin{equation}\label{eq:family invertible at infty}
	B_t\ :=\ B_V\,+\,t\,\Psi_V\,,\qquad\qquad t_0\,\leq\, t\,\leq\,t_1\,,
\end{equation}
where $\Psi_V:=\Psi\tensor \id_V$.
Denote by $H^j$ the $j$-th Sobolev space defined by the operator $B_V$ (see Subsection~\ref{SS:Sobolev}).
We make the following assumption.
\begin{itemize}
	\item[\textbf{(F.1)}] The potential $\Psi_V$ defines an operator in $\mathcal{L}_A(H^j)$, for $j=0,1$.
\end{itemize}
By~\textbf{(F.1)}, $B_t$ and the anticommutator $\left\{B_V,\Psi_V\right\}=B_V\,\Psi_V\,+\,\Psi_V\,B_V$ are in $\mathcal{L}_A(H^1,H^0)$ and $B_t^2$ is in $\mathcal{L}_A(H^2,H^0)$.
We also make the follwing assumption.
\begin{itemize}
	\item[\textbf{(F.2)}] The operators $\{B_t^2\}_{t_0\leq t\leq t_1}$ are \emph{uniformly} invertible at infinity, 
		i.e. there exists a compactly supported smooth function $f:M\rightarrow [0,\infty)$
		such that, for $t_0\leq t\leq t_1$, the operator $B_t^2+f$ is invertible with 
		$\big(B_t^2+f\big)^{-1}\in\mathcal{L}_A(H^0,H^2)$.
\end{itemize}


\begin{proposition}\label{P:stability of the index}
Suppose that Conditions~\textbf{(F.1)} and~\textbf{(F.2)} are satisfied.
Then the class $\ind_AB_t\in K_0(A)$ is independent of $t\in [t_0,t_1]$.
\end{proposition}


Our analysis makes use of the following lemma.


\begin{lemma}[Relllich lemma]\label{L:Rellich lemma}
Let $\nu\in C^\infty_c(M)$. Then the multiplication operator $\nu:H^l\rightarrow H^k$ is $A$-compact for every $k<l$.
\end{lemma}


\noindent This lemma provides a basic set of $A$-compact operators.
In the $C^\ast$-algebra setting, it was first proved in \cite{FM80}. 
We use the formulation given in \cite{Bun95}.


By Condition~\textbf{(F.2)} and Theorem~\ref{T:Bunke}, the class $\ind_AB_t\in K_0(A)$ is represented by the bounded Kasparov module $(H^0,1,F_t)$, where
\begin{equation}\label{eq:operator F_t}
	F_t\ :=\ B_t\,\big(B_t^2+f\big)^{-1/2}
\end{equation}
is defined through Formula~\eqref{eq:integral representation} and where $f$ is as in Condition~\textbf{(F.2)}.
We also consider the operator
\[
	R_t(\lambda)\ :=\ \big(B_t^2+f+\lambda^2\big)^{-1}\,.
\]
Notice that, by~\cite[Lemma~1.5]{Bun95}, $R_t(\lambda)\in\mathcal{L}_A(H^0,H^2)\cap\mathcal{L}_A(H^0)$.


\begin{lemma}\label{L:B_s-B_t}
For all $s,\,t\in[t_0,t_1]$, the difference $B_s-B_t$ is a uniformly bounded bundle map and we have
\begin{equation*}\label{eq:B_s-B_t}
	\|B_s\,-\,B_t\|_\infty\ \leq\ |s-t|\,\left\|\Psi\right\|_\infty\,.
\end{equation*}
\end{lemma}


\begin{proof}
From~\eqref{eq:family invertible at infty}, we have $B_s-B_t=(s-t)\Psi_V$, from which the thesis follows.
\end{proof}


\begin{lemma}\label{L:|R_t(lambda)|}
There exists a constant $d>0$ such that
\begin{equation}\label{eq:|R_lambda(t)|}
	\| R_t(\lambda)\|_{\mathcal{B}(H^0)} \ \leq\  (d+\lambda^2)^{-1}\,,\quad\qquad\qquad  t_0\leq t\leq t_1\,.
\end{equation}
\end{lemma}


\begin{proof}
By~\cite[Lemma~1.5]{Bun95}, we have
\begin{equation}\label{eq:|R_lambda(t)|1}
	\| R_t(\lambda)\|_{\mathcal{B}(H^0)} \ \leq\  \big(d_t+\lambda^2\big)^{-1}\,,
\end{equation}
where
\[
	0\ <\ d_t\ :=\ \inf\left\{\left(\|B_t u\|_0^2+\big\|\sqrt{f} u\big\|_0^2\right): u\in H^2,\|u\|_0=1\right\}\,
\] 
and where $f$ is as in Condition~\textbf{(F.2)}.
Notice that the positivity of $d_t$ is not a trivial fact and is guaranteed by~\cite[Lemma~1.4]{Bun95}.
From~\eqref{eq:|R_lambda(t)|1} we deduce that Inequality~\eqref{eq:|R_lambda(t)|} holds with
\[
	d:=\inf\,\left\{d_t:t_0\leq t\leq t_1\right\}\,.
\]
Finally, from Lemma~\ref{L:B_s-B_t}, it follows that $\{d_t\}_{t_0\leq t\leq t_1}$ varies continuously so that $d$ is strictly positive, which concludes the proof.
\end{proof}


\begin{lemma}\label{L:B_tR_t(lambda)}
There exists a constant $c_1>0$ such that
\begin{equation*}\label{eq:|B_tR_lambda(t)|}
	\left\| B_t\,R_t(\lambda)\right\|_{\mathcal{B}(H^0)} \ \leq\  c_1\,
	\big(d+\lambda^2\big)^{-1/2}\,,\quad\qquad\qquad t_0\leq t\leq t_1\,.
\end{equation*}
\end{lemma}


\begin{proof}
Notice that
\[
	B_t^2\,R_t(\lambda)\ =\ \left[R_t(\lambda)^{-1}-f-\lambda^2\right]\,R_t(\lambda)\ =\ 
	\Id\,-\,\left(f+\lambda^2\right)\,R_t(\lambda)\,.
\]
Using~\eqref{eq:|R_lambda(t)|} and the previous equality, we obtain
\begin{equation*}\label{eq:B_t^2R_lambda(t)}
	\big\|B_t^2\,R_t(\lambda)\big\|_{\mathcal{B}(H^0)}\ \leq\ 1+
	\big\|\big(f+\lambda^2\big)R_t(\lambda)\big\|_{\mathcal{B}(H^0)}\ 
	\leq\ 1+\frac{\|f\|_\infty+\lambda^2}{d+\lambda^2}\ \leq\ c_0\,,
\end{equation*}
for a suitable constant $c_0$.
By the previous inequality and Lemma~\ref{L:|R_t(lambda)|}, for $u\in H^0$ we get
\[
\begin{aligned}
	\| B_t\,R_t(\lambda)u\|_0^2\ =\ &\big|\left<B_t\,R_t(\lambda)u,B_t\,R_t(\lambda)u\right>_0\big|_A\ =\ 
	\big|\left<R_t(\lambda)u,B_t^2\,R_t(\lambda)u\right>_0\big|_A\vspace{0.2cm}\\
	\leq\ &\| R_t(\lambda)u\|_0\,\left\| B_t^2\,R_t(\lambda)u\right\|_0\ \leq\ c_0\,(d+\lambda^2)^{-1}\,\|u\|_0^2\,,
\end{aligned}
\]
from which the thesis follows.
\end{proof}


\begin{lemma}\label{L:B_t^2-B_s^2}
There exists a constant $c_2>0$ such that
\begin{equation*}\label{eq:B_t^2-B_s^2}
	\left\|\left\{B_V,\Psi_V\right\}R_t(\lambda)\right\|_{\mathcal{B}(H^0)}\ \leq\ c_2\,\big(d+\lambda^2\big)^{-1/2}\,,
	\qquad\qquad t_0\leq t\leq t_1\,.
\end{equation*}
\end{lemma}


\begin{proof}
Fix $t\in[t_0,t_1]$ and $u\in H^0$.
By Lemma~\ref{L:|R_t(lambda)|} and Lemma~\ref{L:B_tR_t(lambda)}, we have
\[
\begin{aligned}
	&\left\|R_t(\lambda) u\right\|_1^2\ \leq\ \left\|R_t(\lambda) u\right\|_0^2\,+\,\left\|B_t\,R_t(\lambda) u\right\|_0^2\\
	&\qquad \leq\ \big(d+\lambda^2\big)^{-2}\,\|u\|_0^2\,+\,c_1^2\,\big(d+\lambda^2\big)^{-1}\,\|u\|_0^2
	\ \leq\ \big(d^{-1}+c_1^2\big)\,\big(d+\lambda^2\big)^{-1}\,\|u\|_0^2\,.
\end{aligned}
\]
Hence, 
\[
	\left\|\left\{B_V,\Psi_V\right\}R_t(\lambda)u\right\|_0\ \leq\ \left\|\left\{B_V,\Psi_V\right\}\right\|_{\mathcal{B}(H^1,H^0)}
	\big(d^{-1}+c_1^2\big)^{1/2}\,\big(d+\lambda^2\big)^{-1/2}\,\|u\|_0\,,
\]
from which the thesis follows.
\end{proof}


\subsection{Proof of  Proposition~\ref{P:stability of the index}}\label{SS:Proof of stability of index}
We show that the family $\big\{ F_t\big\}_{t_0\leq t\leq t_1}$ defined by~\eqref{eq:operator F_t} is continuous in $\mathcal{L}_A(H^0)$.
By~\eqref{eq:integral representation}, we deduce
\begin{equation}\label{eq:B_s-B_t integral represantation}
	\big( F_s-F_t\big)\,w\ =\ 
	\frac{2}{\pi}\,\int_0^\infty \left\{B_s\,R_s(\lambda)-B_t\,R_t(\lambda)\right\}w\,d\lambda\,,\qquad\qquad w\in H^1\,.
\end{equation}
We now analyze the integrand term on the right-hand side of~\eqref{eq:B_s-B_t integral represantation}.
We have
\[
	B_t^2\ =\ B_V^2\,+\,t\,\left\{B_V,\Psi_V\right\}\,+\,t^2\Psi_V^2
\]
from which
\[
	B_t^2\,-\,B_s^2\ =\ (t-s)\,\left\{B_V,\Psi_V\right\}\,+\,\left(t^2-s^2\right)\Psi_V^2\,.
\]
Hence,
\begin{equation}\label{eq:B_s-B_t integral represantation1}
\begin{aligned}
	& B_s\,R_s(\lambda)-B_t\,R_t(\lambda)\ =\ 
	\big(B_s-B_t\big)\,R_s(\lambda)+B_t\,\big(R_s(\lambda)-R_t(\lambda)\big)\vspace{0.2cm}\\
	&\quad =\ \big(B_s-B_t\big)\,R_s(\lambda)+B_t\,R_t(\lambda)\,\big(B_t^2-B_s^2\big)\,R_s(\lambda)\vspace{0.2cm}\\
	&\quad =\ (s-t)\left\{\Psi_V\,R_s(\lambda)\,-\,B_t\,R_t(\lambda)\,\left\{B_V,\Psi_V\right\}R_s(\lambda)
	\,-\,B_t\,R_t(\lambda)\,(s+t)\,\Psi_V^2\,R_s(\lambda)\right\}\,.
\end{aligned}
\end{equation}
Using Lemma~\ref{L:|R_t(lambda)|}, Lemma~\ref{L:B_tR_t(lambda)}, and Lemma~\ref{L:B_t^2-B_s^2}, from~\eqref{eq:B_s-B_t integral represantation} and~\eqref{eq:B_s-B_t integral represantation1} we deduce that there exists a constant $c_3>0$ such that
\[
	\big\|\big(F_s-F_t\big)w\big\|_0\ \leq\ \|w\|_0\ |s-t|\ c_3\ \frac{2}{\pi}\,\int_0^\infty\ \frac{d\lambda}{d+\lambda^2}\,,
	\qquad w\in H^1\,.
\]
Since $H^1$ is dense in $H^0$, the last inequality implies the thesis.
\hfill$\square$


\begin{corollary}\label{C:compact case}
Let $P_V$ be a twisted Callias-type operator over a closed manifold $M$. Then the class $\ind_AP_V$ vanishes.
\end{corollary}


\begin{proof}
Suppose $P_V$ is associated to an admissible quadruple $\big(\Sigma,D,\Phi,V\big)$ over $M$.
Since $M$ is closed, using Proposition~\ref{P:stability of the index} we deduce that its index class coincides with the index class of the operator associated to the quadruple $\big(\Sigma,D,-\Phi,V\big)$.
\end{proof}

\section{Reduction to the cylinder}\label{S:reduction}
The next three sections are devoted to the proof of Theorem~\ref{T:Callias-type theorem}.
In this section we consider a twisted Callias-type operator $P_V$ associated to an admissible quadruple $\big(\Sigma,D,\Phi,V\big)$ over a complete \emph{oriented} Riemannian manifold $M$.
We assume there is a partition $M=M_-\cup_NM_+$, where $N=M_-\cap M_+$ is a closed hypersurface and $M_-$ is a compact submanifold with boundary whose interior contains an essential support of $\Phi$.
We reduce the computation of the index class of $P_V$ to the computation of the index class of a \emph{model operator}, i.e. a twisted Callias-type operator ${\bf M}_{V_N}$ on the cylinder $N\times\RR$.
We adapt to the case of an arbitrary $C^\ast$-algebra $A$ the ``cut-and-paste" technique that was used in~\cite{Ang93} and~\cite{Bun95} to prove the case when $A=\CC$ and in~\cite{BC} to prove the case when $A$ is a von Neumann algebra endowed with a finite trace.


\subsection{The model operator}\label{SS:model operator}
Let the twisted Callias-type operator $P_V$, the closed manifold $N$, the $\ZZ_2$-graded Dirac bundle $\Sigma_{N+}=\Sigma_{N+}^+\oplus \Sigma_{N+}^-$ over $N$ and the Hilbert $A$-bundle of finite type $V_N$ over $N$ be as in Subsection~\ref{SS:Callias-type theorem}.
Recall that $V_N$ is endowed with a metric connection $\nabla^{V_N}$.
We use these data to construct a twisted Callias-type operator on the cylinder $N\times\RR$.

Let $p:N\times\RR\rightarrow N$ be the projection onto the first factor and denote by $\widehat{\Sigma}_{N+}$ the pull-back bundle $p^\ast \Sigma_{N+}$. 
Then 
\begin{equation}\label{eq:widehatEN}
	\widehat{\Sigma}_{N+}\ = \ \widehat{\Sigma}_{N+}^+\oplus \widehat{\Sigma}_{N+}^-\,,
	\quad\qquad\text{where}\quad \widehat{\Sigma}_{N+}^\pm\ :=\ p^*\Sigma_{N+}^\pm\,.
\end{equation}
The bundle $\widehat{\Sigma}_{N+}$ has a natural Clifford action given by:
\begin{equation}\label{eq:clifford action}
	\widehat{c}(\xi,t)
	\ = \
	c(\xi)\, + \, \gamma t\,,\qquad  
	(\xi,t)\in T^\ast_{(x,r)} (N\times\RR)\ =\
	T^\ast_x N\oplus\RR,\qquad (x,r)\in N\times \RR\,,
\end{equation}
where $c$ is the Clifford action of $T^*N$ on $\Sigma_{N+}$ and $\gamma=\pm i$ on $\widehat{\Sigma}_{N+}^\pm$. 
Notice, however, that this action does not preserve the grading \eqref{eq:widehatEN}.
Endowed with the pull-back connection $\nabla^{\widehat{\Sigma}_{N+}}$ induced by the connection on $\Sigma_{N+}$, the bundle $\widehat{\Sigma}_{N+}$ becomes an \emph{ungraded} Dirac bundle with associated Dirac operator $\widehat{D}_{N+}$. 

We now define an admissible endomorphism of $\widehat{\Sigma}_{N+}$.
Let $\chi$ be the identity function on $\RR$, i.e. $\chi(r)=r$ for every $r\in\RR$.
With a slight abuse of notation, we denote by $\chi$ also the induced function $N\times\RR\rightarrow\RR$.


\begin{lemma}\label{L:h goes to infinity}
Multiplication by $\chi$ is an admissible endomorphism for the pair $\left(\widehat{\Sigma}_{N+},\widehat{D}_{N+}\right)$ (see Definition~\ref{D:Callias}).
\end{lemma}


\begin{proof}
For a section $u\in C^\infty(N\times \RR,\widehat{\Sigma}_{N+})$, we have
\[
	\left(\big[\widehat{D}_{N+},\chi\big]u\right)(y,r)\ =\ \gamma \,\chi'(r)\,u(y,r)\ =\ \gamma\,u(y,r)\,,
	\qquad\qquad (y,r)\in N\times\RR\,.
\]
Hence, the commutator $\big[\widehat{D}_{N+},\chi\big]$ coincides with the endomorphism $\gamma$ of $\widehat{\Sigma}_{N+}$ and Condition~(i) of Definition~\ref{D:Callias} is satisfied.
Fix a constant $C>0$. 
Then, for $(y,r)\in N\times\RR$, we have
\begin{equation}\label{eq:estimate of chi}
	\chi^2(y,r)\,-\,\big\|\big[\widehat{D}_{N+},\chi\big](y,r)\big\|\ =\ r^2-1\,\geq C\,,\qquad\qquad |r|>\sqrt{C+1}\,.
\end{equation}
Therefore, also Condition~(ii) of Definition~\ref{D:Callias} is satisfied.
\end{proof}


Let $\widehat{V}_N$ be the bundle $V_N$ pulled back to $N\times\RR$.
It is endowed with the connection $\nabla^{\widehat{V}_N}$ obtained by pulling back the connection $\nabla^{V_N}$.


\begin{definition}\label{D:model operator}
The \emph{model operator} on $N\times\RR$ induced by $P_V$ is the twisted Callias-type operator ${\bf M}_{V_N}$ associated to the admissible quadruple $\left(\widehat{\Sigma}_{N+},\widehat{D}_{N+},\chi,\widehat{V}_N\right)$ on $N\times\RR$.
\end{definition}


The next theorem is the main result of this section.
It reduces the computation of the index class of $P_V$ to the computation of the index class of the model operator ${\bf M}_{V_N}$.


\begin{theorem}\label{T:reduction to the cylinder}
The classes $\ind_AP_V$ and $\ind_A{\bf M}_{V_N}$ coincide.
\end{theorem}


\begin{remark}
Notice that the potential $\chi$ is \emph{unbounded}. 
This is the main difference between our  model operator and the one used in~\cite{Ang93} and~\cite{BC}, where the potential $\chi$ is a function that is constant outside of a compact subset of $N\times\RR$.
As already mentioned in Remark~\ref{R:unbounded modules}, our choice of the potential $\chi$ is motivated by the $KK$-theoretical calculations of Section~\ref{S:analysis on the cylinder}.
\end{remark}


The remaining part of this section is devoted to proving Theorem~\ref{T:reduction to the cylinder}.


\subsection{Bunke's relative index theorem}\label{SS:relative index theorem}
In this subsection we review Bunke's $K$-theoretic relative index theorem for the benefit of the reader.
In particular, we formulate this theorem for twisted Callias-type operators.


For $j=0,1$, let $P_j$ be twisted Callias-type operators associated to admissible quadruples $\big(\Sigma_j,D_j,\Phi_j,V_j\big)$ over complete Riemannian manifolds $M_j$.
Suppose $M_j=X_j\cup_{N_j} Y_j$ are partitions of $M_j$, where $N_j=X_j\cap Y_j$ are \emph{closed} hypersurfaces.
We also assume that the quadruples $\big(\Sigma_0,D_0,\Phi_0,V_0\big)$, $\big(\Sigma_1,D_1,\Phi_1,V_1\big)$ \emph{coincide} near $N_0$, $N_1$.
This means that there are tubular neighborhoods $U(N_0)$, $U(N_1)$ respectively of $N_0$, $N_1$ and an isometric diffeomorphism $\psi:U(N_0)\rightarrow U(N_1)$ such that:
\begin{itemize}
	\item $\psi$ restricts to a diffeomorphism between $N_0$ and $N_1$;
	\item there exists an isomorphism of Dirac bundles $\Psi_1:\Sigma_0|_{U(N_0)}\rightarrow \Sigma_1|_{U(N_1)}$ 
		covering $\psi$;
	\item there exists an isomorphism of Hilbert $A$-bundles 
		$\Psi_2:V_0|_{U(N_0)}\rightarrow V_1|_{U(N_1)}$ covering $\psi$ and preserving the connections.
\end{itemize}


Cut $M_j$ along $N_j$ and use the map $\psi$ to glue the pieces together interchanging $Y_0$ and $Y_1$.
In this way we obtain the complete oriented Riemannian manifolds 
\[
	M_2:=X_0\cup_NY_1 \qquad\text{and}\qquad M_3:=X_1\cup_NY_0\,,
\]
where $N\cong N_0\cong N_1$.
We refer to $M_2$ and $M_3$ as the {\em manifolds obtained from $M_0$ and $M_1$ by cutting and pasting}.

Use the map $\Psi_1$ to cut the bundles $\Sigma_0$, $\Sigma_1$ at $N_0$, $N_1$ and glue the pieces together interchanging $\Sigma_0|_{Y_0}$ and $\Sigma_1|_{Y_1}$.
In this way we obtain Dirac bundles $\Sigma_2\rightarrow M_2$ and $\Sigma_3\rightarrow M_3$ with associated Dirac operators $D_2$ and $D_3$.
Define a bundle map $\Phi_2\in C^\infty(M_2;\End(\Sigma_2))$ coinciding with $\Phi_0$ on $X_0$ and with $\Phi_1$ on $Y_1$.
Notice that $\Phi_2$ is admissible for the pair $\big(\Sigma_2,D_2\big)$.
In a similar way, define an admissible endomorphism $\Phi_3$ for the pair $\big(\Sigma_3,D_3\big)$ coinciding with $\Phi_1$ on $X_1$ and with $\Phi_0$ on $Y_0$.
Finally, use the map $\Psi_2$ to cut the bundles $V_0$, $V_1$ at $N_0$, $N_1$ and glue the pieces together interchanging $V_0|_{Y_0}$ and $V_1|_{Y_1}$.
In this way we obtain Hilbert $A$-bundles $V_2\rightarrow M_2$ and $V_3\rightarrow M_3$ endowed with metric connections.

With this procedure, we obtain admissible quadruples
\begin{equation}\label{eq:quadruples cut-paste}
	\big(\Sigma_2,D_2,\Phi_2,V_2\big)\qquad\text{and}\qquad\big(\Sigma_3,D_3,\Phi_3,V_3\big)
\end{equation}
respectively on $M_2$ and $M_3$.
Let $P_2$, $P_3$ be the associated twisted Dirac-type operators.
We refer to~\eqref{eq:quadruples cut-paste} as the \emph{quadruples obtained from $\big(\Sigma_0,D_0,\Phi_0,V_0\big)$ and $\big(\Sigma_1,D_1,\Phi_1,V_1\big)$ by cutting and pasting}.
In this setting, we have the following formulation of Bunke's $K$-theoretic relative index theorem.


\begin{theorem}\label{T:K-theoretic relative index theorem}\emph{(Bunke,~\cite[Theorem~1.2]{Bun95}).}
\[
	\ind_A P_0\,+\,\ind_A P_1\ =\ \ind_A P_2\,+\,\ind_A P_3\,.
\]
\end{theorem}


\subsection{A manifold with the reversed orientation}\label{SS:reverse orientation}
Suppose $\big(\Sigma,D,\Phi,V\big)$ is an admissible quadruple on an complete \emph{oriented} Riemannian manifold $M$.
Let $M^-$ be a copy of this manifold with the opposite orientation.
Denote by $\Sigma^-$ the Dirac bundle $\Sigma$ viewed as a vector bundle over $M^-$, endowed with the {\em opposite} Clifford action. 
This means that a vector $\xi\in T^*M\simeq T^*(M^-)$ acts on $\Sigma^-$ by $c(-\xi)$. 
The change of the Clifford action is needed because we reversed the orientation of $M$.
Denote respectively by $D^-$ and $\Phi^-$ the Dirac-type operator and the endomorphism on $\Sigma^-$ induced by $D$ and $\Phi$.
In this way, we obtain an admissible quadruple $\big(\Sigma^-,D^-,\Phi^-,V^-\big)$ over $M^-$ (cf.~\cite[Chapter~9]{BW93} and~\cite[Section~5]{BC}).
We use this construction and Theorem~\ref{T:K-theoretic relative index theorem} to deduce the following proposition.


\begin{proposition}\label{P:potential off compact}
For $j=0,1$, let $P_j$ be the twisted Callias-type operator associated to an admissible quadruple $\big(\Sigma_j,D_j,\Phi_j,V_j\big)$ over a complete \emph{oriented} Riemannian manifold $M_j$.
Suppose $M_j=X_j\cup_{N_j}Y_j$ is a partition, where $X_j$ is a compact submanifold with boundary whose interior contains an essential support of $\Phi_j$ and $N_j=X_j\cap Y_j$ is a codimension one closed submanifold of $M$.
If the quadruples $\big(\Sigma_0,D_0,\Phi_0,V_0\big)$, $\big(\Sigma_1,D_1,\Phi_1,V_1\big)$ {coincide} near $N_0$, $N_1$, then $\ind_A P_1=\ind_A P_2$. 
\end{proposition}


\begin{proof}
Let $M_0^-$ be a copy of $M$ with the opposite orientation.
By hypothesis, we have the partition $M_0^-=Y_0^-\cup_NX_0^-$, where $N\cong N_0\cong N_1$.
Consider the manifold $M_2:=Y_0^-\cup_N Y_1$.
Let $\big(\Sigma_2,D_2,\Phi_2,V_2\big)$ be the admissible quadruple on $M_2$ coinciding with $\big(\Sigma_0^-,D_0^-,\Phi_0^-,V_0^-\big)$ on $Y_0^-$ and with $\big(\Sigma_0,D_0,\Phi_0,V_0\big)$ on $Y_0$.
Let $P_2$ be the twisted Callias-type operator associated to $\big(\Sigma_2,D_2,\Phi_2,V_2\big)$.
Notice that $\Phi_2$ has empty essential support so that, by Corollary~\ref{C:vanishing sufficient condition}, $\ind_A P_2=0$.

Let $M_3=X_0\cup_N Y_1$, $M_4=Y_0^-\cup_N Y_0$ be the complete Riemannian manifolds obtained from $M_0$ and $M_2$ by cutting and pasting.
Denote by $P_3$ and $P_4$ the twisted Callias-type operators associated respectively to the admissible quadruples $\big(\Sigma_3,D_3,\Phi_3,V_3\big)$ and $\big(\Sigma_4,D_4,\Phi_4,V_4\big)$ obtained from $\big(\Sigma_0,D_0,\Phi_0,V_0\big)$ and $\big(\Sigma_2,D_2,\Phi_2,V_2\big)$ by cutting and pasting.
Notice that $\Phi_4$ has empty essential support so, by Corollary~\ref{C:vanishing sufficient condition}, $\ind_AP_4=0$.
Using Theorem~\ref{T:K-theoretic relative index theorem}, we obtain 
\[
	\ind_A P_0\ = \ \ind_A P_0\ +\ \ind_A P_2\ = \
	\ind_A P_3\ +\ \ind_A P_4\ = \ \ind_AP_3\,.
\]

It remains to show that $\ind_AP_3=\ind_AP_1$.
Consider the manifold $M_5:=X_1\cup_N X_0^-$.
Let $\big(\Sigma_5,D_5,\Phi_5,V_5\big)$ be the admissible quadruple on $M_5$ coinciding with $\big(\Sigma_1,D_1,\Phi_1,V_1\big)$ on $X_1$ and with $\big(\Sigma_0^-,D_0^-,\Phi_0^-,V_0^-\big)$ on $X_0^-$.
Denote by $P_5$ the associated twisted Callias-type operator.
Since $M_5$ is a closed manifold, by Corollary~\ref{C:compact case}, $\ind_A P_5=0$.

Let $M_6=X_3\cup_N Y_1$, $M_7=Y_3^-\cup_N Y_3$ be the complete Riemannian manifolds obtained from $M_3$ and $M_5$ by cutting and pasting.
Denote by $P_6$ and $P_7$ the twisted Callias-type operators associated respectively to the admissible quadruples $\big(\Sigma_6,D_6,\Phi_6,V_6\big)$ and $\big(\Sigma_7,D_7,\Phi_7,V_7\big)$ obtained from $\big(\Sigma_3,D_3,\Phi_3,V_3\big)$ and $\big(\Sigma_5,D_5,\Phi_5,V_5\big)$ by cutting and pasting.
Notice that $M_7=M_1$ and that the quadruples $\big(\Sigma_7,D_7,\Phi_7,V_7\big)$ and $\big(\Sigma_1,D_1,\Phi_1,V_1\big)$ coincide.
Hence, $\ind_AP_7=\ind_AP_1$.
Moreover, by Corollary~\ref{C:compact case} $\ind_AP_6=0$.
Using Theorem~\ref{T:K-theoretic relative index theorem},  we finally deduce
\[
	\ind_A P_3\ = \ \ind_A P_3\ +\ \ind_A P_5\ = \
	\ind_A P_6\ +\ \ind_A P_7\ = \ \ind_AP_7\ =\ \ind_A P_1\,.
\]
\end{proof}


\subsection{Reduction to a manifold with cylindrical ends}\label{SS:cylindrical ends}
Let $P_V$, $M$, $\big(\Sigma,D,\Phi,V\big)$ and the partition $M=M_-\cup_NM_+$ be as in Subsection~\ref{SS:Callias-type theorem}.
Deform the metric, the Clifford bundle structure, the operator $D$, the bundle map $\Phi$ and the Hilbert $A$-bundle $V$ in such a way that there exist $\epsilon>0$ and a neighborhood $U(N)$ of $N$ satisfying
\begin{itemize}
	\item $U(N)$ is diffeomorphic to $N\times(1-\epsilon,1+\epsilon)$ and the Riemannian metric, restricted to $U(N)$,
		has a product structure; 
	\item the restriction of $\big(\Sigma,D,\Phi,V\big)$ to $N\times(1-\epsilon,1+\epsilon)$ coincides with 
		$\big(\widehat{\Sigma}_N,\widehat{D}_N,\widehat{\Phi}_N,\widehat{V}_N\big)$, where 
		$\widehat{\Sigma}_N,\widehat{D}_N,\widehat{\Phi}_N,\widehat{V}_N$ are defined in Subsection~\ref{SS:model operator};
	\item the set $M_-\setminus \left(N\times(1-\epsilon, 1]\right)$ is an essential support of $\Phi$.
\end{itemize}
For more details on this deformation, cf.~\cite[Section~6]{BC}.
Since all the changes occur in a compact set, the $A$-index of $P_V$ is stable under such deformation.

Consider the manifold with cylindrical ends
\begin{equation}\label{eq:manifold M_C}
	M_C\ :=\ M_-\cup_N\big(N\times[1,\infty)\big)\,.
\end{equation}
Let $\big(\Sigma_C,D_C,\Phi_C,V_C\big)$ be the quadruple coinciding with $\big(\Sigma,D,\Phi,V\big)$ on $M_-\cup_N\big(N\times [1,1+\epsilon)\big)$ and with $\big(\widehat{\Sigma}_N,\widehat{D}_N,\widehat{\Phi}_N,\widehat{V}_N\big)$ on $N\times (1-\epsilon,\infty)$.
Notice that this quadruple is admissible and the set $M_-\setminus \left(N\times(1-\epsilon, 1]\right)$ is an essential support of $\Phi_C$.
Denote by $P_C$ the twisted Callias-type operator associated with the quadruple $\big(\Sigma_C,D_C,\Phi_C,V_C\big)$.
The following proposition is a direct consequence of Proposition~\ref{P:potential off compact}.


\begin{proposition}\label{P:reduction to cylindrical ends}
The classes $\ind_AP_V$ and $\ind_AP_C$ coincide.
\end{proposition}


\subsection{A perturbation of the connection on the cylindrical end}\label{SS:perturbed connection}
Let the manifold $M_C$ and the quadruple $\big(\Sigma_C,D_C,\Phi_C,V_C\big)$ be as in Subsection~\ref{SS:cylindrical ends}.
In this subsection we modify the connection $\nabla^{\Sigma_C}$ of $\Sigma_C$ over the cylindrical end.
The goal is to get a new connection which preserves the grading.

Let us first introduce some notations.
Let $\nabla^{\Sigma_N}$ and $\nabla^{\Sigma_{N\pm}}$ be the connections respectively on $\Sigma_N$ and $\Sigma_{N\pm}$ defined in Subsection~\ref{SS:Callias-type theorem}.
Denote by $\nabla^{\widehat{\Sigma}_N}$ and $\nabla^{\widehat{\Sigma}_{N\pm}}$ the lifts of these connections respectively to the bundles $\widehat{\Sigma}_N$ and $\widehat{\Sigma}_{N\pm}$.
Notice that in general $\nabla^{\widehat{\Sigma}_N}\neq \nabla^{\widehat{\Sigma}_{N+}}\oplus\nabla^{\widehat{\Sigma}_{N-}}$.

Over the cylindrical end, we have the decomposition
\begin{equation}\label{eq:decomposition of Eigma_C}
	\Sigma_C\big|_{N\times(1-\epsilon,\infty)}\ =\ \widehat{\Sigma}_{N+}\oplus\widehat{\Sigma}_{N-}\,.
\end{equation}
Notice that the connection $\nabla^{\Sigma_C}$ of $\Sigma_C$ doesn't preserve Decomposition~\eqref{eq:decomposition of Eigma_C}.
In particular, with respect to such decomposition, we have
\begin{equation}\label{eq:reduction to the cylinder2}
    {D_C}\big|_{N\times (1-\epsilon,\infty)}\ =\
    \begin{pmatrix}
    		\widehat{D}_{N+}&
    		\widehat{\pi}_+\circ \widehat{D}_N\circ \widehat{\pi}_-
    		\vspace{0.2cm}\\
		\widehat{\pi}_-\circ \widehat{D}_N\circ \widehat{\pi}_+& 
		\widehat{D}_{N-}
    \end{pmatrix}\,,
\end{equation}
where $\widehat{\pi}_\pm$ are the projections onto $\widehat{\Sigma}_{N\pm}$.
In~\cite[Subsection~5.16]{BC}, it is shown that the operators $\widehat{\pi}_\pm\circ \widehat{D}_N\circ \widehat{\pi}_\mp$ are of order zero.
Define a bundle map $\Pi\in C^\infty(M_C;\End(\Sigma_C))$ that is $0$ outside of $N\times (1-\epsilon,\infty)$ and such that
\begin{equation}\label{eq:Psi}
	\Pi\big|_{N\times (1-\epsilon_1,\infty)}\ :=\
 	\begin{pmatrix}
		0&\widehat{\pi}_+\circ \widehat{D}_N\circ \widehat{\pi}_-\vspace{0.2cm}\\
		\widehat{\pi}_-\circ \widehat{D}_N\circ \widehat{\pi}_+& 0
 	\end{pmatrix}\,,
\end{equation}
where $0<\epsilon_1<\epsilon$.

Set $D_C^\prime:= D_C-\Pi$.
Notice that
\[
	D_C^\prime\big|_{N\times(1-\epsilon_1,\infty)}\ =\ \widehat{D}_{N+}\oplus\widehat{D}_{N-}\,.
\]
Hence, we regard the operator $D_C^\prime$ as a formally self-adjoint Dirac-type operator associated to a new connection $\nabla^{\Sigma_C^\prime}$ on $\Sigma_C$, coinciding with $\nabla^{\Sigma_C}$ outside of $N\times (1-\epsilon,\infty)$ and such that
\begin{equation}\label{eq:split connection}
	\nabla^{\Sigma_C^\prime}\big|_{N\times(1-\epsilon_1,\infty)}\ =\ 
	\nabla^{\widehat{\Sigma}_{N+}}\oplus \nabla^{\widehat{\Sigma}_{N-}}\,.
\end{equation}
Denote by $\Sigma^\prime_C$ the Dirac bundle $\Sigma_C$ endowed with the connection $\nabla^{\Sigma_C^\prime}$.


\begin{lemma}\label{L:ind_AP_C'}
There exists $\lambda\geq 1$ such that the quadruple $\big(\Sigma_C^\prime,D_C^\prime,\lambda\,\Phi_C,V_C\big)$ is admissible.
Moreover, if $P_C^\prime$ is the twisted Callias-type operator associated with this quadruple, then the classes $\ind_AP_C^\prime$ and $\ind_AP_C$ coincide.
\end{lemma}


\begin{proof}
Since both endomorphisms, $\Pi$ and $\Phi_C$, are uniformly bounded, the commutator $[\Pi,\Phi_C]$ belongs to $L^\infty(M;\End(\Sigma_C))$.
Since the restriction of $D_C$ to $N\times (1-\epsilon,\infty)$ is the lift of $D_N$, the commutator $[D,\Phi_C]$ is also in $L^\infty(M;\End(\Sigma_C))$.
Choose constants $d>0$ and $\lambda\geq 1$ such that 
\begin{equation}\label{eq:admissible lambda}
	\lambda^2\Phi^2_C(x)\ \geq\ d\,+\, \lambda\left(\|[D_C,\Phi_C]\|_\infty\,+\,\|\Pi,\Phi_C\|_\infty\right)\,,
	\qquad\qquad x\in N\times (1-\epsilon_1,\infty)\,.
\end{equation}
Observe that, for all $t\in [0,1]$, we have
\[
	\lambda\left(\|[D_C,\Phi_C]\|_\infty+\|\Pi,\Phi_C\|_\infty\right)
	\ \geq\ \|[D_C\,-\,t\,\Pi,\lambda\,\Phi_C](x)\|\,,\quad\qquad x\in N\times (1-\epsilon_1,\infty)\,.
\]
Using~\eqref{eq:admissible lambda}, from the previous inequality we deduce
\begin{equation}\label{eq:admissible lambda1}
	\left(\lambda\,\Phi_C\right)^2(x)\,-\,\left\|\left[D_C-t\,\Pi,\lambda\,\Phi_C\right]\right\|\ \geq\ d\,,
	\qquad\qquad x\in N\times (1-\epsilon_1,\infty),\quad 0\leq t\leq 1\,.
\end{equation}
It follows that the quadruple $\big(\Sigma_C,D_C-t\Pi,\lambda\Phi_C,V_C\big)$ is admissible, for $0\leq t\leq 1$, and we denote by $P_C^{\lambda t}$ the associated twisted Callias-type operator.
By Proposition~\ref{P:stability of the index}, the classes $\ind_A P_C^{\lambda 0}$ and $\ind_A P_C^{\lambda 1}$ coincide.

Since $D_C^\prime=D_C-\Pi$, the quadruple $\big(\Sigma_C^\prime,D_C^\prime,\lambda\,\Phi_C,V_C\big)$ is admissible and the associated operator $P_C^\prime$ has the same index class as $P_C^{\lambda 1}$.
Finally, using again Proposition~\ref{P:stability of the index}, the classes $\ind_A P_C$ and $\ind_A P_C^{\lambda 0}$ coincide.
Therefore,
\[
	\ind_AP_C\ =\ \ind_AP_C^{\lambda 0}\ =\ \ind_AP_C^{\lambda 1}\ =\ \ind_AP_C^\prime\,,
\]
from which the thesis follows.
\end{proof}


\subsection{A perturbation of the potential on the cylindrical end}\label{SS:perturbed potential}
In this subsection we modify the potential on the cylindrical end.
The goal is to obtain a new potential that goes to infinity at infinity.

Let $\zeta:\RR\rightarrow \RR$ be an increasing smooth function such that there are positive constants $R$, $c$ and $\epsilon_2$ satisfying the following conditions.
\begin{enumerate}
	\item[\textbf{(Z.1)}] $0<\epsilon_2<\epsilon_1$ and
		\[
		\zeta(r)\ =\ \left\{\begin{array}{cl}
		0\,,&\qquad\qquad 0< r< 1-\epsilon_1\\	
		1\,,&\qquad\qquad 1-\epsilon_2< r< 1+\epsilon_2\\
		r\,,&\qquad\qquad |r|>R 
		\end{array}\right.\,,
		\]
		where $\epsilon_1$ is the constant of~\eqref{eq:Psi};
	\item[\textbf{(Z.2)}] $\zeta^2(r)\,-\,\|\zeta^\prime(r)\|\ \geq\ c\,,\  $ for $r>1-\epsilon_2$.
\end{enumerate}
With a slight abuse of notation, we denote by $\zeta$ also the induced function $N\times\RR\rightarrow\RR$.
Let $\Phi_C^{\prime\prime}$ be the endomorphism of $\Sigma_C^\prime$ vanishing outside of $N\times (1-\epsilon_1,\infty)$ and such that $\Phi_C^{\prime\prime}=\zeta\,\widehat{\alpha}$ on the cylindrical end $N\times(1-\epsilon,\infty)$.
Here, $\widehat{\alpha}$ is the grading operator of $\widehat{\Sigma}_N$, i.e. $\widehat{\alpha}=\pm 1$ on $\widehat{\Sigma}_{N\pm}$.
By Condition~\textbf{(Z.1)}, the endomorphism $\Phi_C^{\prime\prime}$ is admissible for the pair $\big(\Sigma_C^\prime,D_C^\prime\big)$.
By Condition~\textbf{(Z.2)}, the set $M\setminus\big(N\times(1-\epsilon_2,\infty)\big)$ is an essential support of $\Phi_C^{\prime\prime}$.
Let $P_C^{\prime\prime}$ be the twisted Callias-type operator associated to the quadruple $\big(\Sigma_C^\prime,D_C^\prime,\Phi_C^{\prime\prime},V_C\big)$.


\begin{lemma}\label{L:ind_AP_C''}
The classes $\ind_AP_C^\prime$ and $\ind_AP_C^{\prime\prime}$ coincide.
\end{lemma}


\begin{proof}
Let $f_1:\RR\rightarrow [0,1]$ be a smooth function with support in $(1-\epsilon_1,\infty)$ such that $f(r)=\zeta(r)$ for $0<r<1$ and  $f_1(r)=1$ for $r\geq 1$.
Regard $f_1$ as a function $M_C\rightarrow [0,1]$ and define the endomorphism $\Phi_C^{\prime\prime\prime}:=f_1\,\widehat{\alpha}$.
Notice that $\Phi_C^{\prime\prime\prime}$ is admissible for the pair $\big(\Sigma_C^\prime,D_C^\prime\big)$ and the set $M\setminus\big(N\times(1-\epsilon_2,\infty)\big)$ is an essential support of $\Phi_C^{\prime\prime\prime}$.
Denote by $P_C^{\prime\prime\prime}$ the twisted Callias-type operator associated to the quadruple $\big(\Sigma_C^\prime,D_C^\prime,\Phi_C^{\prime\prime\prime},V_C\big)$.
By Proposition~\ref{P:potential off compact}, the classes $\ind_AP_C^{\prime\prime\prime}$ and $\ind_AP_C^{\prime\prime}$ coincide. Therefore, to prove the thesis it suffices to show that $\ind_AP_C^{\prime\prime\prime}=\ind_AP_C^\prime$.

Consider the family of potentials
\[
	\Phi_t\ :=\ t\Phi_C^{\prime\prime\prime}\,+\,(1-t)\,\Phi_C^\prime\,,\qquad\qquad 0\leq t\leq1\,.
\]
In~\cite[Subsection~5.17]{BC}, it is shown that there exists $\mu\geq\lambda$ such that, for all $0\leq t\leq 1$, the endomorphism $\mu\,\Phi_t$ is admissible for the pair $\big(\Sigma_C^\prime,D_C^\prime\big)$, with essential support independent of $t$.
Notice that $\Phi_0=\Phi_C^\prime$ and $\Phi_1=\Phi_C^{\prime\prime\prime}$.
Let $P_{\mu C}^\prime$ and $P_{\mu C}^{\prime\prime\prime}$ be the twisted Callias-type operator associated respectively to the admissible quadruples $\big(\Sigma_C^\prime,D_C^\prime,\mu\,\Phi_C^\prime,V_C\big)$ and  $\big(\Sigma_C^\prime,D_C^\prime,\mu\,\Phi_C^{\prime\prime\prime},V_C\big)$.
Since the endomorphisms $\mu\Phi_C^\prime$ and $\mu\Phi_C^{\prime\prime\prime}$ are constant on the cylindrical end, the classes $\ind_AP_{\mu C}^\prime$ and $\ind_AP_{\mu C}^{\prime\prime\prime}$ coincide by Proposition~\ref{P:stability of the index}.

Moreover, using Proposition~\ref{P:stability of the index} again, we deduce that the class $\ind_AP_C^\prime$ coincides with $\ind_AP_{\lambda C}^\prime$ and the class $\ind_AP_{\mu C}^{\prime\prime\prime}$ coincides with $\ind_AP_C^{\prime\prime\prime}$.
Therefore,
\[
	\ind_AP_C^\prime\ =\ \ind_AP_{\mu C}^\prime\ =\ \ind_AP_{\mu C}^{\prime\prime\prime}\ =\ 
	\ind_AP_C^{\prime\prime\prime}\,,
\]
from which the thesis follows.
\end{proof}


\subsection{Proof of Theorem~\ref{T:reduction to the cylinder}}
From Proposition~\ref{P:reduction to cylindrical ends}, Lemma~\ref{L:ind_AP_C'} and Lemma~\ref{L:ind_AP_C''}, it suffices to show that $\ind_A{\bf M}_{V_N}=\ind_A P_C^{\prime\prime}$, where $P_C^{\prime\prime}$ is the twisted Callias-type operator associated to the admissible quadruple $\big(\Sigma_C^\prime,D_C^\prime,\Phi_C^{\prime\prime},V_C\big)$ defined in Subsection~\ref{SS:perturbed potential}.

Let $M_C^-$ be a copy of $M_C$ with the reversed orientation.
From the construction of Subsection~\ref{SS:reverse orientation},  the quadruple $\big(\Sigma_C^\prime,D_C^\prime,\Phi_C^{\prime\prime},V_C\big)$ induces an admissible quadruple $\big(\Sigma_C^-,D_C^-,\Phi_C^-,V_C\big)$ on $M_C^-$.
Moreover, from~\eqref{eq:manifold M_C} we have the partition $M_C^-=\big(N\times (-\infty,1]\big)\cup_NM_-^-$, where $M_-^-$ is a copy of $M_-$ with reversed orientation and where we identify $N\times (-\infty,1]$ with a copy of $N\times [1,\infty)$ with reversed orientation.

Let us construct an admissible quadruple on the complete Riemannian manifold
\[
	M_1\ :=\ \big(N\times (-\infty,1]\big)\cup_N(M^-_-)\,.
\]
Let $\Sigma_1\rightarrow M_1$ be the Dirac bundle coinciding with $\widehat{\Sigma}_N$ on $N\times(-\infty,1+\epsilon_2)$ and with $\Sigma_C^-$ on $\big(N\times (1-\epsilon_2,1]\big)\cup_N(M_-^-)$.
Define a Dirac-type operator $D_1\in\Diff^1(M_1;\Sigma_1)$ through the conditions $D_1=\widehat{D}_{N+}\oplus\widehat{D}_{N-}$ on $N\times(-\infty,1+\epsilon_2)$ and $D_1=D_C^-$ on $\big(N\times (1-\epsilon_2,1]\big)\cup_N(M_-^-)$.
Let $\Phi_1\in C^\infty(M_1;\End(\Sigma_1))$ be the endomorphism coinciding with $\Phi_C^-$ on $\big(N\times (1-\epsilon_2,1]\big)\cup_N(M_-^-)$ and such that
\[
	\Phi_1\big|_{N\times(-\infty,1+\epsilon_2)}\ =\ \left\{
	\begin{array}{cc}\zeta\,,&\text{ on } \widehat{\Sigma}_{N+}\vspace{0.2cm}\\-1\,,&\text{ on } \widehat{\Sigma}_{N-}\end{array}\right.\,.
\]
Observe that
\[
	[D_1,\Phi_1]\big|_{N\times(-\infty,1+\epsilon_2)}\ =\ 
	\left[\begin{pmatrix}\widehat{D}_{N+}&0\vspace{0.2cm}\\ 0&\widehat{D}_{N-}\end{pmatrix},
	\begin{pmatrix}\zeta&0\vspace{0.2cm}\\0&-1\end{pmatrix}\right]\ =\ 
	\begin{pmatrix}[\widehat{D}_{N+},\zeta]&0\vspace{0.2cm}\\0&0\end{pmatrix}\,.
\]
Thus, the endomorphism $\Phi_1$ is admissible for the pair $\big(\Sigma_1,D_1\big)$.
Finally, let $V_1\rightarrow M_1$ be the Hilbert $A$-bundle such that $V_1=V_C$ on $\big(N\times (1-\epsilon_2,1]\big)\cup_N(M^-_-)$ and $V_1=\widehat{V}_N$ on ${N\times(-\infty,1+\epsilon_2)}$.
Denote by $P_1$ the twisted Callias-type operator associated to the admissible quadruple $\big(\Sigma_1,D_1,\Phi_1,V_1\big)$.
Since the endomorphism $-\id_{\Sigma_1}$ has empty essential support and $\Phi_1=-\id_{\Sigma_1}$ outside of a compact set, $\ind_AP_1=0$ by Corollary~\ref{C:vanishing sufficient condition} and Proposition~\ref{P:potential off compact}.

Notice that the manifolds $M_C$, $M_1$ and the quadruples $\big(\Sigma_C^\prime,D_C^\prime,\Phi_C^{\prime\prime},V_C\big)$, 
$\big(\Sigma_1,D_1,\Phi_1,V_1\big)$ coincide near $N$ in the sense of Subsection~\ref{SS:relative index theorem}.
Let
\[
	M_2\ =\ M_-\cup (M^-_-)\qquad\text{and}\qquad M_3\ =\ N\times\RR
\]
be the complete Riemannian manifolds obtained by $M_C$ and $M_1$ by cutting and pasting.
Let $P_2$ and $P_3$ be the twisted Callias-type operators associated respectively to the admissible quadruples $\big(\Sigma_2,D_2,\Phi_2,V_2\big)$ and $\big(\Sigma_3,D_3,\Phi_3,V_3\big)$, obtained from $\big(\Sigma_C^\prime,D_C^\prime,\Phi_C^{\prime\prime},V_C\big)$ and  $\big(\Sigma_1,D_1,\Phi_1,V_1\big)$ by cutting and pasting.

Notice that we have the decomposition $\Sigma_3=\widehat{\Sigma}_{N+}\oplus\widehat{\Sigma}_{N-}$.
Notice also that, with respect to this decomposition, we have the splitting
\[
	\big(\Sigma_3,D_3,\Phi_3,V_3\big)\ =\ \big(\widehat{\Sigma}_{N+},\widehat{D}_{N+},\zeta,\widehat{V}_{N+}\big)
	\oplus \big(\widehat{\Sigma}_{N-},\widehat{D}_{N-},-1,\widehat{V}_{N-}\big)\,.
\]
Therefore,
\[
	\ind_AP_3\ =\ \ind_AT_++\ind_AT_-\,,
\]
where $T_+$ and $T_-$ are the twisted Callias-type operators associated respectively to the admissible quadruples 
$\big(\widehat{\Sigma}_{N+},\widehat{D}_{N+},\zeta,\widehat{V}_{N+}\big)$ and $\big(\widehat{\Sigma}_{N-},\widehat{D}_{N-},-1,\widehat{V}_{N-}\big)$ over $N\times\RR$.

By Proposition~\ref{P:potential off compact}, the classes $\ind_A{\bf M}_{V_N}$ and $\ind_AT_+$ coincide.
Since the endomorphism $-\id_{\widehat{\Sigma}_{N-}}$ has empty essential support, the class $\ind_AT_-$ vanishes by Corollary~\ref{C:vanishing sufficient condition}.
Moreover, the manifold $M_2$ is compact so that the class $\ind_AP_2$ vanishes by Corollary~\ref{C:compact case}.
Using Theorem~\ref{T:K-theoretic relative index theorem}, we finally deduce
\[
	\ind_AP_C^{\prime\prime}\ =\ \ind_AP_C^{\prime\prime}\,+\,\ind_AP_1\ =\ 
	\ind_AP_2\,+\,\ind_AP_3\ =\ T_+\,+\,T_-\ =\ \ind_A{\bf M}_{V_N}\,,
\]
from which the thesis follows.
\hfill$\square$

\section{An unbounded Kasparov module}\label{S:growing potentials}
We use the model operator ${\bf M}_{V_N}$ to define an \emph{unbounded} Kasparov module representing the class $\ind_A{\bf M}_{V_N}$.
This fact will allow us to use unbounded $KK$-theory to do the computations on the cylinder in Section~\ref{S:analysis on the cylinder}.


\subsection{Unbounded Kasparov modules}\label{SS:unbounded modules}
Let $\mathcal{A}$ and $\mathcal{B}$ be graded $C^\ast$-algebras.
\begin{definition}\label{D:unbounded modules}
An \emph{unbounded Kasparov module} for $(\mathcal{A},\mathcal{B})$ is a triple $(E,\phi,D)$, where $E$ is a $\ZZ_2$-graded Hilbert $\mathcal{B}$-module, $\phi\colon \mathcal{A}\rightarrow\mathcal{L}_\mathcal{B}(E)$ is a graded $\ast$-homomorphism, and $D$ is an odd self-adjoint regular operator on $E$ such that
\begin{enumerate}
	\item $(1+D^2)^{-1}\phi(a)$ extends to an element of $\mathcal{K}_\mathcal{B}(E)$ for every $a\in \mathcal{A}$;
	\item the set of $a\in \mathcal{A}$ such that $[D,\phi(a)]$ is densely defined and extends to an 
		element of $\mathcal{L}_\mathcal{B}(E)$ 
		is dense in $\mathcal{A}$.
\end{enumerate}
\end{definition}

The relationship between bounded and unbounded Kasparov modules is clarified by the next proposition.


\begin{proposition}[Baaj-Julg,~\cite{BJ83}]\label{P:unbounded to bounded}
Given an \emph{unbounded} Kasparov module $(E,\phi,D)$ for $(\mathcal{A},\mathcal{B})$, the triple $(E,\phi,D(D^2+1)^{-1/2})$ is a \emph{bounded} Kasparov module for $(\mathcal{A},\mathcal{B})$.
In this case, we say that the $KK$-theoretical element
\[
	\big[E,\phi,D(D^2+1)^{-1/2}\big]\in KK(\mathcal{A},\mathcal{B})
\]
is the class defined by the unbounded Kasparov module $(E,\phi,D)$.
\end{proposition}


\subsection{An alternative definition of the $A$-index of the model operator}\label{SS:A-index of M}
Let the closed manifold $N$, the $\ZZ_2$-graded Dirac bundle $\Sigma_{N+}=\Sigma_{N+}^+\oplus \Sigma_{N+}^-$ and the Hilbert $A$-bundle of finite type $V_N$ be as in Section~\ref{SS:Callias-type theorem}.
Let ${\bf M}_{V_N}$ be the \emph{model operator} associated to these data.
It is a twisted Callias-type operator on the cylinder $N\times \RR$ (see Subsection~\ref{SS:model operator}). 
In this section, we denote by $H^j$ the $j$-th Sobolev space defined by the operator ${\bf M}_{V_N}$ (see Subsection~\ref{SS:Sobolev}).

\begin{theorem}\label{T:unbounded kasparov module}
The triple 
\begin{equation}\label{eq:unbounded A-ind}
	\left(H^0\big(\widehat{V}_N\tensor\widehat{\Sigma}_{N+}\big)\oplus 
	H^0\big(\widehat{V}_N\tensor\widehat{\Sigma}_{N+}\big),1,
	{\bf M}_{V_N}\right)
\end{equation}
is an unbounded Kasparov module for the pair of algebras $(\CC,A)$.
Here, $1$ denotes complex scalar multiplication.
\end{theorem}


\begin{proof}
Condition~(2) of Definition~\ref{D:unbounded modules} is trivially satisfied.
Let us verify Condition~(1).
Let $h:N\times\RR\rightarrow \RR$ be the function defined by setting $h(x,r)=r^2-1$ for $(x,r)\in N\times\RR$.
Notice that $h$ is a coercive function, i.e. $h$ is smooth, proper and bounded from below (sse~\cite[Definition~2.12]{Ebe16}).
By~\eqref{eq:estimate of chi}, ${\bf M}_{V_N}^2\geq h$.
By~\cite[Theorem~3.40]{Ebe16}, it follows that ${\bf M}_{V_N}$ has compact resolvent.
Hence, Condition~(1) of Definition~\ref{D:unbounded modules} is also satisfied.
\end{proof}


\begin{theorem}\label{T:equivalence of definitions}
Let $f:M\rightarrow [0,\infty)$ be a compactly supported smooth function such that ${\bf M}_{V_N}^2+f$ is invertible and $\big({\bf M}_{V_N}^2+f\big)^{-1}$ is in $ \mathcal{L}_A(H^0,H^2)$.
Then the operator
\[
	{\bf M}_{V_N}\big({\bf M}_{V_N}^2+1\big)^{-1/2}\,-\,{\bf M}_{V_N}\big({\bf M}_{V_N}^2+f\big)^{-1/2}
\]
is $A$-compact.
\end{theorem}


Before proving Theorem~\ref{T:equivalence of definitions}, we deduce the following consequence.


\begin{corollary}
The class in $K_0(A)$ defined by the unbounded Kasparov module~\eqref{eq:unbounded A-ind} coincides with $\ind_A{\bf M}_{V_N}$.
\end{corollary}


\begin{proof}
Let $f$ be as in the hypothesis of Theorem~\ref{T:equivalence of definitions}.
By Theorem~\ref{T:Bunke}, the class $\ind_A{\bf M}_{V_N}$ is represented by the bounded Kasparov module
\begin{equation}\label{eq:bounded A-ind1}
	\left(H^0\big(\widehat{V}_N\tensor\widehat{\Sigma}_{N+}\big)\oplus 
	H^0\big(\widehat{V}_N\tensor\widehat{\Sigma}_{N+}\big),1,
	{\bf M}_{V_N}\big({\bf M}_{V_N}^2+f\big)^{-1/2}\right)\,.
\end{equation}
By Proposition~\ref{P:unbounded to bounded}, the class defined by the unbounded Kasparov module~\eqref{eq:unbounded A-ind} is the element of $K_0(A)$ represented by the bounded Kasparov module
\begin{equation}\label{eq:unbounded A-ind1}
	\left(H^0\big(\widehat{V}_N\tensor\widehat{\Sigma}_{N+}\big)\oplus 
	H^0\big(\widehat{V}_N\tensor\widehat{\Sigma}_{N+}\big),1,
	{\bf M}_{V_N}\big({\bf M}_{V_N}^2+1\big)^{-1/2}\right)\,.
\end{equation}
Now the thesis follows from~\eqref{eq:bounded A-ind1}, ~\eqref{eq:unbounded A-ind1} and Theorem~\ref{T:equivalence of definitions}.
\end{proof}


The remaining part of this section is devoted to the proof of Theorem~\ref{T:equivalence of definitions}.


\begin{lemma}\label{L:compact integrand}
The integral
\begin{equation}\label{eq:compact operator}
	\frac{2}{\pi}\int_0^\infty\left\{{\bf M}_{V_N}\,\big({\bf M}_{V_N}^2+1+\lambda^2\big)^{-1}\,-\,
	{\bf M}_{V_N}\, \big({\bf M}_{V_N}^2+f+\lambda^2\big)^{-1}\right\}\,d\lambda
\end{equation}
converges in operator norm and defines an element in $\mathcal{K}_A(H^0)$.
\end{lemma}


\begin{proof}
By Proposition~\ref{P:invertibility of B_V^2+lambda^2}, the operator $ \big({\bf M}_{V_N}^2+1+\lambda^2\big)^{-1}$ is in $\mathcal{L}_A(H^0,H^2)$ and, by Theorem~\ref{T:B^2 is invertible at infinity} and~\cite[Lemma~1.5]{Bun95}, the operator $ \big({\bf M}_{V_N}^2+f+\lambda^2\big)^{-1}$ is also in $\mathcal{L}_A(H^0,H^2)$.
Hence, the operator 
\begin{equation}\label{eq:integrandum operator}
	{\bf M}_{V_N}\,\big({\bf M}_{V_N}^2+1+\lambda^2\big)^{-1}\,-\,
	{\bf M}_{V_N}\, \big({\bf M}_{V_N}^2+f+\lambda^2\big)^{-1}
\end{equation} 
is in $\mathcal{L}_A(H^0)$.
We have
\begin{equation}\label{eq:estimate of integrand}
	\begin{array}{l}
	{\bf M}_{V_N}\, \big({\bf M}_{V_N}^2+1+\lambda^2\big)^{-1}\,-\,
		{\bf M}_{V_N}\, \big({\bf M}_{V_N}^2+f+\lambda^2\big)^{-1}\vspace{0.2cm}\\ 
	\qquad\qquad =\ {\bf M}_{V_N}\,\left\{ \big({\bf M}_{V_N}^2+1+\lambda^2\big)^{-1}-\, 
		\big({\bf M}_{V_N}^2+f+\lambda^2\big)^{-1}\right\}\vspace{0.2cm}\\
	\qquad\qquad =\ {\bf M}_{V_N}\, \big({\bf M}_{V_N}^2+f+\lambda^2\big)^{-1}\,(f-1)\, 
		\big({\bf M}_{V_N}^2+1+\lambda^2\big)^{-1}\,.
	\end{array}
\end{equation}
Since, by ~\cite[Theorem~3.40]{Ebe16}, the operator $ \big({\bf M}_{V_N}^2+1+\lambda^2\big)^{-1}$ is in $\mathcal{K}_A(H^0)$, the previous calculation shows that the operator~\eqref{eq:integrandum operator} is $A$-compact.
Moreover, by~\cite[Lemma~1.5]{Bun95},~\cite[Lemma~1.6]{Bun95}  and~\eqref{eq:estimate of integrand}, there exist positive constants $c$ and $d$ such that
\[
	\left\|{\bf M}_{V_N}\, \big({\bf M}_{V_N}^2+1+\lambda^2\big)^{-1}\,-\,
	{\bf M}_{V_N}\, \big({\bf M}_{V_N}^2+f+\lambda^2\big)^{-1}\right\|_{\mathcal{B}(H^0)}\ \leq\ c\,(d+\lambda^2)^{-1}\,.
\]
Therefore, the integral~\eqref{eq:compact operator} converges in operator norm and defines an element in $\mathcal{K}_A(H^0)$.
\end{proof}


\subsection{Proof of Theorem~\ref{T:equivalence of definitions}}
Fix $w\in H^1$. 
By~\cite[Lemma~1.8]{Bun95}, we have
\[
	{\bf M}_{V_N}\, \big({\bf M}_{V_N}^2+1\big)^{-1/2}w\ =\ 
	\frac{2}{\pi}\int_0^\infty{\bf M}_{V_N}\, \big({\bf M}_{V_N}^2+1+\lambda^2\big)^{-1}\,w\,d\lambda\,,
\]
where the integral converges in norm.
Using~\eqref{eq:integral representation} , we obtain
\begin{eqnarray*}
	&&{\bf M}_{V_N}\, \big({\bf M}_{V_N}^2+1\big)^{-1/2}w\,-\,{\bf M}_{V_N}\, 
		\big({\bf M}_{V_N}^2+f\big)^{-1/2}w\ =\vspace{0.2cm}\\
	&&\qquad
		=\ \frac{2}{\pi}\,\int_0^\infty\left\{{\bf M}_{V_N}\,\big({\bf M}_{V_N}^2+1+\lambda^2\big)^{-1}\,-\,
		{\bf M}_{V_N}\, \big({\bf M}_{V_N}^2+f+\lambda^2\big)^{-1}\right\}w\,d\lambda\,.
\end{eqnarray*}
Now the thesis follows from Lemma~\ref{L:compact integrand} and the density of $H^1$ in $H^0$.
\hfill$\square$

\section{Analysis on the cylinder}\label{S:analysis on the cylinder}
We complete the proof of Theorem~\ref{T:Callias-type theorem}.
We solve a model problem on a cylinder $N\times\RR$ with compact base.
Then we use results from Section~\ref{S:reduction} to deduce the general case.
The computations on the cylinder make heavy use of the properties of the intersection product in $KK$-theory and of the notion of connection for unbounded Kasparov modules developed by Kucerovsky.
In particular, we adapt some $KK$-theoretical computations of~\cite{Kuc01} to the case of operators twisted with Hilbert $C^\ast$-bundles.


\subsection{The setting}\label{SS:model problem}
Let the closed manifold $N$, the $\ZZ_2$-graded Dirac bundle $\Sigma_{N+}=\Sigma_{N+}^+\oplus \Sigma_{N+}^-\rightarrow N$ with associated Dirac operator $D_{N+}$ and the Hilbert $A$-bundle of finite type $V_N$ over $N$ be as in Subsection~\ref{SS:Callias-type theorem}.
Denote by $D_{N+,V_N}$ the operator obtained by twisting $D_{N+}$ with the bundle $V_N$.
Its index class $\ind_AD_{N+,V_N}\in K_0(A)$ is defined in~\cite{FM80}.


The next theorem is the main result of this section.
It allows us to reduce the computation of the $A$-index of the model operator ${\bf M}_{V_N}$ to the computation of the index class of $D_{N+,V_N}$.


\begin{theorem}\label{T:analysis on the cylinder}
	Suppose that the $C^\ast$-algebra $A$ is separable.
	Then the classes $\ind_A{\bf M}_{V_N}$ and $\ind_A D_{N+,V_N}$ coincide.
\end{theorem}


\begin{remark}
Notice that, by setting $M_-:=N\times[-1,1]$ and $M_+:=(N\times(-\infty,-1])\sqcup (N\times [1,\infty))$, this theorem is a particular instance of Theorem~\ref{T:Callias-type theorem}.
\end{remark}


\subsection{Proof of Theorem~\ref{T:Callias-type theorem}}
By Theorem~\ref{T:reduction to the cylinder} and Theorem~\ref{T:analysis on the cylinder}, we obtain
\[
	\ind_A P_V\ =\ \ind_A {\bf M}_{V_N}\ =\ \ind_A D_{N+,V_N}.
\]
\hfill$\square$


The remaining part of this section is devoted to proving Theorem~\ref{T:analysis on the cylinder}.
We first review some background material.


\subsection{Intersection product}\label{SS:intersection product}
Let $C_1$ be the complex Clifford algebra associated to $\CC$ and the canonical nondegenerate quadratic form.
The algebra $C_1$ is naturally endowed with a $\ZZ_2$-grading, where the even (resp. odd) elements are of the form $z\oplus z$ (resp. $z\oplus -z$), for some $z\in \CC$ (cf.~\cite[Examples~14.1.2.(b)]{Bla98}).
For $n\in\ZZ_+$, the groups $KK^n(\mathcal{A},\mathcal{B})$ are defined by setting $KK^0(\mathcal{A},\mathcal{B}) := KK(\mathcal{A},\mathcal{B})$ and $KK^{j+1}(\mathcal{A},\mathcal{B})=KK^j(\mathcal{A},\mathcal{B}\widehat{\tensor} C_1)$.
By Bott periodicity (\cite[Corollary~17.8.9]{Bla98}), $KK^0(\mathcal{A},\mathcal{B})\cong KK^0(\mathcal{A} \widehat{\tensor}C_1,\mathcal{B}\widehat{\tensor}C_1)$ and $KK^1(\mathcal{A},\mathcal{B})\cong KK^0(\mathcal{A} \widehat{\tensor}C_1,\mathcal{B})$.

The $KK$-theory groups are endowed with two products, i.e. the \emph{composition product}
\begin{equation}\label{eq:composition product}
	\circ:KK(\mathcal{A},\mathcal{B})\times KK(\mathcal{B},\mathcal{C})\xrightarrow{\qquad} KK(\mathcal{A},\mathcal{C})
\end{equation}
and the \emph{exterior product}
\begin{equation}\label{eq:exterior product}
	\widehat{\tensor}:KK(\mathcal{A}_1,\mathcal{B}_1)\times KK(\mathcal{A}_2,\mathcal{B}_2)\xrightarrow{\qquad} 
	KK(\mathcal{A}_1\widehat{\tensor} \mathcal{B}_1,\mathcal{A}_2\widehat{\tensor} \mathcal{B}_2)\,,
\end{equation}
where we use the (graded) minimal spacial tensor product (cf.~\cite[Appendix~T]{Weg93}).
In order to make these products defined, the $C^\ast$-algebras need to satisfy some conditions.
For the product~\eqref{eq:composition product} we require $\mathcal{A}$ separable and $\mathcal{B}$ $\sigma$-unital.
For~\eqref{eq:exterior product} we require that $\mathcal{A}_1$ and $\mathcal{A}_2$ are separable and $\mathcal{B}_1$ is $\sigma$-unital (see~\cite[Section~18.9]{Bla98}).
In the next proposition, we recall some properties of the products~\eqref{eq:composition product} and~\eqref{eq:exterior product} that will be used in this section.
We will assume that all $C^\ast$-algebras satisfy the appropriate size restrictions (separable or $\sigma$-unital) necessary to make the products defined.


\begin{proposition}\label{P:intersection product}(Properties of the intersection product)
For ${\bf x}\in KK(\mathcal{A},\mathcal{B})$, ${\bf y}\in KK(\mathcal{B},\mathcal{C})$ and ${\bf z}\in KK(\mathcal{C},\mathcal{D})$, we have
\begin{enumerate}[label=(\alph*)]
\item $\left({\bf x}\circ{\bf y}\right)\circ{\bf z}\,=\,{\bf x}\circ\left({\bf y}\circ{\bf z}\right)$;
\item $\left({\bf y}\circ{\bf z}\right)\widehat{\tensor}[\id_\mathcal{A}] \ =\ \left({\bf y}\widehat{\tensor}[\id_\mathcal{A}]\right)\circ\left({\bf z}\widehat{\tensor}[\id_\mathcal{A}]\right)$.
\end{enumerate}
Here,	the element $[\id_\mathcal{A}]\in KK(\mathcal{A},\mathcal{A})$ is given by the triple $(\mathcal{A},1,0)$.
\end{proposition}


\subsection{Connections for unbounded modules}\label{SS:unbounded connections}
Connes and Skandalis~\cite{CS84} defined the notion of connection to give a criterion under which an element ${\bf z}\in KK(\mathcal{A},\mathcal{C})$ is expressable as composition product ${\bf x}\circ{\bf y}={\bf z}$, for some ${\bf x}\in KK(\mathcal{A},\mathcal{B})$ and ${\bf y}\in KK(\mathcal{B},\mathcal{C})$.
Their construction is based on the expression of the $KK$-elements as \emph{bounded} Kasparov modules.
Kucerovsky~\cite{Kuc97} extended the notion of connection to the case when the elements are expressed as \emph{unbounded} Kasparov modules.
We quickly review this approach.


\begin{definition}
Let $S$, $T$ be unbounded operators on a Hilbert $\mathcal{A}$-module $E$.
We say that the resolvent of $T$ is \emph{compatible} with $S$ if there is a dense submodule $\mathcal{W}$ of $E$ such that the operator $S(i\mu+T)^{-1}(i\mu_1+S)^{-1}$ is defined on $\mathcal{W}$,  for all $\mu$, $\mu_1\in\RR\setminus\{0\}$.
\end{definition}


The next proposition is the main technical tool that we use to do computations with unbounded Kasparov modules.


\begin{proposition}[Kucerovsky,~\cite{Kuc97}]\label{P:unbounded product}
Suppose that the classes ${\bf x}\in KK(\mathcal{A},\mathcal{B})$, ${\bf y}\in KK(\mathcal{B},\mathcal{C})$ and ${\bf z}\in KK(\mathcal{A},\mathcal{C})$ are represented respectively by the unbounded Kasparov modules $(E_1,\phi_1,D_1)$, $(E_2,\phi_2,D_2)$ and $(E,\phi_1\widehat{\tensor} 1,D)$, where $E\cong E_1\widehat{\tensor}_{\phi_1} E_2$. 
For $x\in E_1$, let $T_x\in\mathcal{L}_\mathcal{C}(E_2,E)$ be the operator defined by setting $T_x(y):=x \widehat{\tensor} y$.
Moreover, suppose that the following conditions are satisfied:
\begin{enumerate}[label=(\roman*)]
	\item (Connection condition) the operator $DT_x-(-1)^{\partial x}T_x D_2$  is bounded on $\Dom(D_2)$
		for all homogeneous $x$ in some dense subset of $\phi_1(A) E_1$;
	\item (Compatibility condition) either the resolvent of $(D_1\widehat{\tensor}1)$ is compatible with $D$  or 
		the resolvent of $D$ is compatible with $(D_1\widehat{\tensor}1)$;
	\item (Positivity condition) the graded commutator 
	$[D,D_1\widehat{\tensor}1]=D\,(D_1\widehat{\tensor}1)+(D_1\widehat{\tensor}1)\,D$ 
	is bounded below on a dense submodule of $E$.
\end{enumerate}
Then ${\bf x}\circ{\bf y}={\bf z}$.
\end{proposition}


\subsection{Twisted Dirac operators on compact manifolds}
The Hilbert $A$-bundle $V_N\rightarrow N$ defines an element $[V_N]\in KK(\CC,C(N;A))$ represented by the triple $\left(C(N;V_N),1,0\right)$, where $1$ denotes scalar multiplication.
The Dirac operator $D_{N+}$ defines an element $[D_{N+}]\in KK(C(N),\CC)$ through the unbounded Kasparov module
\begin{equation}
	\left(H^0\left(N;\Sigma_{N+}^+\right)\oplus H^0\left(N;\Sigma_{N+}^-\right),1,
	\left(\begin{array}{cc}
	0&D_{N+}^-\\D_{N+}^+&0
	\end{array}\right)
	\right)\,,
\end{equation}
where $1$ denotes pointwise multiplication.
The relationship between the $A$-index of $D_{N+,V_N}$ and the elements $[V_N]$ and $[D_{N+}]$ is given in the next proposition (for the proof we refer to~\cite[Section~5.3]{Sch05}).


\begin{proposition}\label{P:twisted Dirac}
	$\ind_A D_{N+,V_N}\ =\ [V_N]\circ\left([D_{N+}]\widehat{\tensor}[\id_A]\right)$.
\end{proposition}


\subsection{KK-theoretical version of Anghel's theorem}
The operator $\widehat{D}_{N+}$ defines an element $\big[\widehat{D}_{N+}\big]$ in $KK^1(C_0(N\times\RR),\CC)$ through the unbounded Kasparov module
\begin{equation}
	\left(H^0\big(\widehat{\Sigma}_{N+}\big)\oplus H^0\big(\widehat{\Sigma}_{N+}\big),\phi,
	\left(\begin{array}{cc}
	0&\widehat{D}_{N+}\\\widehat{D}_{N+}&0
	\end{array}\right)
	\right)
\end{equation}
for $\big(C_0(N\times\RR)\widehat{\tensor} C_1,\CC\big)$.
Here, $\phi:C_0(N\times\RR)\widehat{\tensor} C_1\rightarrow \mathcal{L}\big(H^0\big(\widehat{\Sigma}_{N+}\big)\oplus H^0\big(\widehat{\Sigma}_{N+}\big)\big)$ is the graded homomorphism defined as follows.
By~\cite[Corollary~14.5.3]{Bla98}, even (respectively odd) elements of $C_0(N\times\RR)\widehat{\tensor} C_1$ are of the form $b\oplus b$ (resp. $b\oplus-b$) for some $b\in C_0(N\times\RR)$.
Then the homomorphism $\phi$ is given by
\[
	\phi:b\oplus b\mapsto\left(\begin{array}{cc}b&0\\0&b\end{array}\right)
	\qquad\qquad
	\phi:b\oplus -b\mapsto\left(\begin{array}{cc}0&-ib\\ib&0\end{array}\right)\,,
\]
for $b\in C_0(N\times\RR)$.

The function $\chi$ defined in Subsection~\ref{SS:model operator} gives an element $[\chi]\in KK^1(C(N),C_0(N\times \RR))$ through the Kasparov module
\begin{equation}
	\left(C_0(N\times\RR)\widehat{\tensor} C_1,\psi,
	\left(\begin{array}{cc}\chi&0\\0&-\chi
	\end{array}\right)
	\right)
\end{equation}
for $\big(C(N),C_0(N\times\RR)\widehat{\tensor} C_1\big)$.
Here, the homomorphism $\psi:C(N)\rightarrow \mathcal{L}_{C_0(N\times\RR)\widehat{\tensor} C_1}(C_0(N\times\RR)\widehat{\tensor} C_1)$ is defined as follows.
Given a function $u\in C(N)$, let $\widehat{u}\in C(N\times \RR)$ be the lift of $u$ to $N\times \RR$, i.e. $\widehat{u}(x,r)=u(x)$ for $(x,r)\in N\times\RR$.
Then $\psi(u)$ is pointwise multiplication by $\widehat{u}$.
The next proposition is the $KK$-theoretical version the Callias-type theorem due to Anghel~\cite{Ang93}. 


\begin{proposition}[Kucerovsky,~\cite{Kuc01}]\label{P:Anghel theorem}
	$[D_{N+}]\ =\ [\chi]\circ\big[\widehat{D}_{N+}\big]$.	
\end{proposition}


\subsection{$A$-index of Twisted Callias-type operators}
In order to prove Theorem~\ref{T:analysis on the cylinder}, we need to connect the $A$-index of the model operator ${\bf M}_{V_N}$ with the elements $[V_N]$, $\big[\widehat{D}_{N+}\big]$ and $[\chi]$.

The product $\big[\widehat{D}_{N+}\big]\widehat{\tensor}[\id_A]\in KK^1(C_0(N\times\RR)\widehat{\tensor} A,A)$ is given by the triple
\begin{equation}	
	\left(\left(H^0\left(\widehat{\Sigma}_{N+}\right)\oplus H^0\left(\widehat{\Sigma}_{N+}\right)\right)\widehat{\tensor} A,
	\phi\widehat{\tensor}\id_A,
	\left(\begin{array}{cc}
	0&\widehat{D}_{N+}\widehat{\tensor} \id_A\\\widehat{D}_{N+}\widehat{\tensor}\id_A&0
	\end{array}\right)
	\right)\,.
\end{equation}
Moreover, the product $[\chi]\widehat{\tensor} [\id_A]\in KK^1(C(N;A),C_0(N\times \RR; A))$ is represented by the unbounded Kasparov module
\begin{equation}
	\left(C_0(N\times\RR;A)\widehat{\tensor} C_1,\psi\widehat{\tensor} \id_A,
	\left(\begin{array}{cc}\chi&0\\0&-\chi
	\end{array}\right)
	\right)
\end{equation}
for the pair of algebras $\big(C(N;A),C_0(N\times\RR;A)\widehat{\tensor} C_1\big)$.
Here, we used the isomorphisms 
\[
	C(N)\,\widehat{\tensor}\, A\cong C(N;A)\qquad\text{and}
	\qquad C_0(N\times\RR)\,\widehat{\tensor}\, A\cong C_0(N\times\RR;A)\,.
\]


\begin{lemma}\label{L:structure of the model operator1}
The element $[V_N]\circ\left([\chi]\widehat{\tensor} [\id_A]\right)\in KK^1(\CC,C_0(N\times \RR; A))$ is represented by the unbounded Kasparov module
\begin{equation}
	\left(C_0(N\times\RR;\widehat{V}_N)\widehat{\tensor} C_1,1,
	\left(\begin{array}{cc}\chi&0\\0&-\chi
	\end{array}\right)
	\right)
\end{equation}
for the pair of algebras $\big(\CC,C_0(N\times\RR;A)\widehat{\tensor} C_1\big)$, where $1$ denotes complex scalar multiplication.
\end{lemma}


\begin{proof}
By~\cite[Corollary~14.5.3]{Bla98}, even (resp. odd) elements of $C_0(N\times\RR;A)\,\widehat{\tensor}\, C_1$ are of the form $g\oplus g$ (resp. $g\oplus -g$), for some $g\in C_0(N\times\RR;A)$.
Fix $g\in C_0(N\times\RR;\widehat{V}_N)$ and $u\in C(N;V_N)$.
We have
\[
\begin{aligned}
	\left(\begin{array}{cc}\chi &0\\0&-\chi\end{array}\right)\,T_u(g\oplus \pm g)\ -\ & 
	T_u \left(\begin{array}{cc}\chi&0\\0&-\chi\end{array}\right)(g\oplus \pm g)\ =\vspace{0.2cm}\\ 
	=\ &
	\left(\begin{array}{cc}\chi &0\\0&-\chi\end{array}\right)\left(\begin{array}{c}u\tensor g\\u\tensor \pm g\end{array}\right)-
	T_u \left(\begin{array}{c}\chi g\\\mp \chi g\end{array}\right)\ =\ 0\,,
\end{aligned}
\]
from which Condition (i) of Proposition~\ref{P:unbounded product} follows.
Conditions~(ii) and (iii) are trivially verified.
\end{proof}


The next lemma provides the wanted connection between the $A$-index of ${\bf M}_{V_N}$ and the cycles defined in the previous two subsections.
When $A=\CC$, this lemma follows from~\cite[Lemma~3.1]{Kuc01}.
 
\begin{lemma}\label{L:structure of model operator}
$\ind_A{\bf M}_{V_N}\ =\ \left\{[V_N]\circ\left([\chi] \,\widehat{\tensor}\, [\id_A]\right)\right\}\circ\left(\big[\widehat{D}_{N+}\big] \,\widehat{\tensor}\, [\id_A]\right)$.
\end{lemma}


\begin{proof}
Set
\[\begin{aligned}
	& E_1\ :=\ C_0(N\times\RR;\widehat{V}_N) \,\widehat{\tensor}\, C_1\,,\qquad
	E_2\ :=\ \left(H^0\left(\widehat{\Sigma}_{N+}\right)\oplus H^0\left(\widehat{\Sigma}_{N+}\right)\right)\widehat{\tensor} A\,,
	\vspace{0.2cm}\\
	& \qquad\qquad
	E\ :=\ H^0\big(\widehat{V}_N\tensor \widehat{\Sigma}_{N+}\big)\oplus 
	H^0\big(\widehat{V}_N\tensor \widehat{\Sigma}_{N+}\big)\,.
\end{aligned}\]
Let us first prove that $E_1\widehat{\tensor}_\phi E_2\cong E$.
We need to show that
\begin{equation}\label{eq:E_1tensorE_2}
	\left< \alpha\tensor \beta,\gamma\tensor \delta\right>_E\ =\ 
	\left< \alpha\tensor \beta,\gamma\tensor \delta\right>_{E_1\widehat{\tensor}_\phi E_2},
\end{equation}
for every homogeneous $\alpha$, $\gamma\in E_1$ and every $\beta$, $\delta\in E_2$.
Suppose first that $\partial \alpha=\partial\gamma=0$, i.e. $\alpha=b\oplus b$ and $\gamma=c\oplus c$, for some $b,\,c\in C_0(N\times\RR;\widehat{V}_N)$.
Suppose also that 
\[
	\beta\ =\ \left(\begin{array}{c}p\\q\end{array}\right)\qquad\qquad\delta\ =\ \left(\begin{array}{c}r\\s\end{array}\right),
\]
for some $p,q,r,s\in H^0(\widehat{\Sigma}_{N+})\,\widehat{\tensor}\, A$.
Then
\[
\begin{aligned}
	&\left< \alpha\tensor \beta,\gamma\tensor \delta\right>_E\ =\ 
		\left< (b\oplus b)\tensor \left(\begin{array}{c}p \\q\end{array}\right), 
		(c\oplus c)\tensor\left(\begin{array}{c}r \\s\end{array}\right)\right>_E
		\vspace{0.2cm}\\
	&=\,  \int_{N\times\RR}\left(\left<b(x),c(x)\right>_{\widehat{V}_{N,x}}\left<p(x),r(x)\right>_{\widehat{\Sigma}_{N+,x}}
		+\left<b(x),c(x)\right>_{\widehat{V}_{N,x}}\left<q(x),s(x)\right>_{\widehat{\Sigma}_{N+,x}}\right)d\text{vol}(x)	
		\vspace{0.2cm}\\
	&=\,  \int_{N\times\RR}\left(\left<p(x),r(x)
		\left<b(x),c(x)\right>_{\widehat{V}_{N,x}}\right>_{\widehat{\Sigma}_{N+,x}\tensor A}
		+\left<q(x),s(x)\left<b(x),c(x)\right>_{\widehat{V}_{N,x}}\right>_{\widehat{\Sigma}_{N+,x}\tensor A}\right)
		d\text{vol}(x)
		\vspace{0.2cm}\\
	&=\, \left< \left(\begin{array}{c}p \\q\end{array}\right), 
		\phi\Big(\left< b\oplus b,c\oplus c\right>_{E_1}\Big)
		\left(\begin{array}{c}r \\s\end{array}\right)\right>_{E_2}
		\ =\ \left< (b\oplus b)\tensor \left(\begin{array}{c}p \\q\end{array}\right),
		(c\oplus c)\tensor\left(\begin{array}{c}r \\s\end{array}\right)\right>_{E_1\widehat{\tensor}_\phi E_2}\,.
\end{aligned}
\]
Hence, Equation~\eqref{eq:E_1tensorE_2} holds when $\partial \alpha=\partial\gamma=0$. 
The cases when either $\partial \alpha=1$ or $\partial\gamma=1$ are obtained with similar computations.


To prove the thesis, it remains to verify Conditions (i), (ii) and (iii) of Proposition~\ref{P:unbounded product}.


(\emph{Connection condition}).
Set
\[
\begin{aligned}
	&D_1\ :=\ \left(\begin{array}{cc}\chi&0 \\0&-\chi\end{array}\right)\qquad\qquad\qquad
	D_2\ :=\ \left(\begin{array}{cc}
	0&\widehat{D}_{N+}\widehat{\tensor} \id_A\\\widehat{D}_{N+}\widehat{\tensor}\id_A&0
	\end{array}\right)\vspace{0.2cm}\\	
	&\qquad\qquad D\ :=\ \left(\begin{array}{cc}0&\widehat{D}_{N+,V_N}-i\chi \\\widehat{D}_{N+,V_N}+i\chi&0\end{array}\right)\,.
\end{aligned}
\]
Let $w$ be a homogeneous element in $C^\infty_c(N\times\RR;\widehat{V}_N) \,\widehat{\tensor}\, C_1$.
We need to show that the operator $T_w D_2 -(-1)^{\partial w}D T_w$ is bounded on $\Dom(D_2)$.
Let us consider first the case when $\partial w=1$, i.e. $w=g\oplus -g$, for some $g\in C^\infty_c(N\times\RR;\widehat{V}_N)$.
For $u\tensor a,\,v\tensor b\in H^1\big(\widehat{\Sigma}_{N+}\big)\widehat{\tensor} A$, we have
\[
	T_{g\oplus -g}\, D_2\left(\begin{array}{c}u\tensor a \vspace{0.2cm}\\v\tensor b\end{array}\right)\ =\ 
	T_{g\oplus -g}\left(\begin{array}{c}\widehat{D}_{N+}v\tensor b \vspace{0.2cm}
	\\\widehat{D}_{N+}u\tensor a\end{array}\right)\ =\ 
	i\,\left(\begin{array}{c}-\widehat{D}_{N+}u \tensor ga \vspace{0.2cm}\\ \widehat{D}_{N+}v\tensor gb\end{array}\right)
\]
and
\[
	D\,T_{g\oplus -g} \left(\begin{array}{c}u\tensor a \vspace{0.2cm}\\v\tensor b\end{array}\right)\ =\ 
	\left(\begin{array}{c}i\widehat{D}_{N+,V_N}(u\tensor ga) \vspace{0.2cm}\\
	-i\widehat{D}_{N+,V_N}(v\tensor gb)\end{array}\right)\,+\,
	\left(\begin{array}{c} \chi\, u\tensor ga\vspace{0.2cm}\\\chi\, v\tensor gb\end{array}\right)
\]
Moreover, 
\[
	\widehat{D}_{N+,V_N}(s\tensor g)\ =\ \big(\widehat{D}_{N+}s\big)\tensor g+(D_g s)\,,
	\qquad\qquad s\in H^1\big(\widehat{\Sigma}_{N+}\big)\,.
\]
Here, $D_g$ is the operator expressed in local coordinates as
\[
	D_g(s)
	\ =\ \sum_i\left(c(X^i) s\right)\tensor \left(\nabla_{X_i}^{V_N} g\right)\,,
\]
where $\{X_i\}$ is a local orthonormal frame of $T(N\times\RR)$ and $X^i$ is the dual frame of $T^*(N\times\RR)$. 
Therefore,
\begin{equation}\label{eq:connection condition}
	(T_{g\oplus -g}\, D_2\,+\,D\,T_{g\oplus -g}) \left(\begin{array}{c}u\tensor a \vspace{0.2cm}\\v\tensor b\end{array}\right)\ =\ 
	\left(\begin{array}{c}iD_g(u)a\vspace{0.2cm}\\-iD_g(v)b\end{array}\right)\,+\,
	\left(\begin{array}{c} \chi\, u\tensor ga\vspace{0.2cm}\\\chi\, v\tensor gb\end{array}\right)\,.
\end{equation}
Since $g$ is compactly supported, the operators $s\mapsto \chi\,s\tensor g$ and $s\mapsto D_g(s)$ are bounded.
By Equation~\ref{eq:connection condition}, the operator $T_{g\oplus -g}\, D_2\,+\,D\,T_{g\oplus -g}$ is bounded as well.
With a similar reasoning, it is proved that the operator $T_{g\oplus g}\, D_2\,-\,D\,T_{g\oplus g}$ is bounded, from which Condition~(i) of  Proposition~\ref{P:unbounded product} follows.


(\emph{Compatibility condition}).
Under the isomorphism $E_1\widehat{\tensor}_\phi E_2\cong E$, the domain of $D$ is contained in the domain of $D_1\widehat{\tensor}1$.
Hence, by~\cite[Lemma~10]{Kuc97} the resolvent of $D$ is compatible with $D_1\widehat{\tensor}1$ and  Condition~(ii) of  Proposition~\ref{P:unbounded product} holds.


(\emph{Positivity condition}).
We have
\[
	\left[D,D_1\widehat{\tensor} 1\right]\ =\ 
	\left(\begin{array}{cc}2\chi^2+i[\widehat{D}_{N+,V_N},\chi]&0\vspace{0.2cm}\\0&2\chi^2-i[\widehat{D}_{N+,V_N},\chi]\end{array}\right)\,.
\]
By Inequality~\eqref{eq:estimate of chi}, the operator on the right-hand side of the last equality is bounded below on $C^\infty_c\big(\widehat{V}_N\tensor\big(\widehat{\Sigma}_{N+}\oplus \widehat{\Sigma}_{N+}\big)\big)$, from which Condition~(iii) of  Proposition~\ref{P:unbounded product} follows.
\end{proof}


\subsection{Proof of Theorem~\ref{T:analysis on the cylinder}}
We have:
\[
\begin{aligned}
	\ind_A{\bf M}_{V_N}\
	=\ & \left\{[V_N]\circ\left([\chi] \,\widehat{\tensor}\,  [\id_A]\right)\right\}
	\circ\left(\big[\widehat{D}_{N+}\big] \,\widehat{\tensor}\,  [\id_A]\right)&
	\quad\qquad\text{by Lemma~\ref{L:structure of model operator}}\vspace{0.2cm}\\
	=\ & [V_N]\circ\left\{\left([\chi] \,\widehat{\tensor}\,  [\id_A]\right)
	\circ\left(\big[\widehat{D}_{N+}\big] \,\widehat{\tensor}\,  [\id_A]\right)\right\}&
	\qquad\text{by Proposition~\ref{P:intersection product}.(a)}\vspace{0.2cm}\\
	=\ & [V_N]\circ\left\{\left([\chi]\circ\big[\widehat{D}_{N+}\big] \right)
	 \,\widehat{\tensor}\,  [\id_A]\right\}&
	\text{by Proposition~\ref{P:intersection product}.(b)}\vspace{0.2cm}\\
	=\ & [V_N]\circ\left([D_{N+}] 
	 \,\widehat{\tensor}\,  [\id_A]\right)&\text{by Lemma 1}\vspace{0.2cm}\\
	=\ & \ind_AD_{N+,V_N}&\text{by Proposition~\ref{P:twisted Dirac}}\,.
\end{aligned}
\]
\hfill$\square$

\begin{remark}\label{R:algebra is separable}
The assumption that $A$ is separable in part (b) of Theorem~\ref{T:Theorem B} is due to the previous computation and the size restriction on $C^\ast$-algebras we made in Subsection~\ref{SS:intersection product}.
\end{remark}


\section{The vanishing theorem}\label{S:vanishing}
This last section is devoted to the proof of Theorem~\ref{T:codimension one obstructions} and Theorem~\ref{T:Theorem A}.


\subsection{Bochner-Lichnerowitz Formula}
Let $(M,g)$, $\kappa$, $V$ and $\nabla^V$ be as in Subsection~\ref{SS:codimension 1 obstructions}.
Let $\mathbb{S}_M$ be the spinor bundle over $M$.
Since the dimension of $M$ is odd, $\mathbb{S}_M$ is an \emph{ungraded} Dirac bundle.
Denote by $\slashed{D}_M$ the associated Dirac operator.
On the bundle $\mathbb{S}_M\tensor V$ consider the connection $\nabla^{\mathbb{S}_M\tensor V}:=\nabla^{\mathbb{S}_M}\tensor 1+1\tensor \nabla^{V}$ and the associated Laplace operator
\[
	\Delta_{\mathbb{S}_M\tensor V}\ :=\ \big(\nabla^{\mathbb{S}_M\tensor V}\big)^\ast \circ\nabla^{\mathbb{S}_M\tensor V}\,,
\]
where $\big(\nabla^{\mathbb{S}_M\tensor V}\big)^\ast$ is the formal adjoint of $\nabla^{\mathbb{S}_M\tensor V}$.


\begin{proposition}[Bochner-Lichnerowitz]\label{P:B-L formula} 
$\slashed{D}_{M,V}^2\ =\ \Delta_{\mathbb{S}_M\tensor V}+\frac{1}{4}\kappa$.
\end{proposition}


\begin{remark}
In Proposition~\ref{P:B-L formula}, the hypothesis that the connection $\nabla^V$ is flat is crucial.
In fact, if $\nabla^V$ is not flat, on the right-hand side of this formula it appears a remainder term depending on the curvature of $\nabla^V$ (see~\cite[Theorem~II.8.17]{LM89} and the following remark).
\end{remark}


\subsection{Proof of Theorem~\ref{T:codimension one obstructions}}
Let $h:M\rightarrow [0,\infty)$ be a smooth function satisfying $h=1$ on $M_+$ and $h=-1$ on $M_-\setminus L$, where $L\subset M_-$ is a compact neighborhood of $\partial M_-$ such that $k$ is strictly positive on $L$.
Notice that for every $\lambda>0$ the function $\lambda h$ is an admissible endomorphism for the pair $\left(\mathbb{S}_M,\slashed{D}_M\right)$.
Denote by $P_\lambda$ the twisted Callias-type operator associated to the admissible quadruple $(\mathbb{S}_M,\slashed{D}_M,\lambda h,V)$.
Notice that the operator $D_{N+,V_N}$ induced by these data on $N$ (see Subsection~\ref{SS:Callias-type theorem}) coincides with the opeartor $\slashed{D}_{N,V_N}$.
Hence, by Theorem~\ref{T:Callias-type theorem} we deduce
\begin{equation}\label{eq:P_lambda=D_N}
	\ind_A P_\lambda \ =\  \ind_A \slashed{D}_{N,V_N}\,.
\end{equation}
From Proposition~\ref{P:B-L formula}, we obtain
\[
	P_\lambda^2\ =\ \left(\begin{array}{cc}
	\Delta_{\mathbb{S}_M\tensor V}+\frac{1}{4}\kappa\,+\,\lambda^2 h^2\,+\,i\big[\slashed{D}_{M,V},h\big]&\vspace{0.2cm}\\	
	0& \Delta_{\mathbb{S}_M\tensor V}+\frac{1}{4}\kappa\,+\,\lambda^2 h^2\,-\,i\big[\slashed{D}_{M,V},h\big]
	\end{array}\right)\,.
\]
For $\lambda$ small enough, $\frac{\kappa}{4}+\lambda^2-\lambda |h^\prime|\geq c$, for some positive constant $c$.
Hence, by Corollary~\ref{C:condition for ind_AB_V=0}, the class $\ind_A P_\lambda$ vanishes.
Now the thesis follows from Equation~\eqref{eq:P_lambda=D_N}
\hfill$\square$


\subsection{Higher Dirac obstructions on closed spin manifolds}\label{SS:higher obstructions}
Let $X$ be a spin manifold with associated spin-Dirac operator $\slashed{D}_X$.
Let $\pi$ be the fundamental group of $X$ and let $\widetilde{X}$ be its universal cover.
Denote by $C^\ast_\CC\pi$ the complex (reduced or maximal) group $C^\ast$-algebra of $\pi$.
The Hilbert $C^\ast_\CC\pi$-bundle $\mathcal{V}(X)\ :=\ \widetilde{X}\times_\pi C^\ast_\CC\pi$ is called the \emph{Mi\v{s}\v{c}enko-Fomenko line bundle}.
It is endowed with a canonical flat connection.
Denote by $\slashed{D}_{X,\mathcal{V}(X)}$ the Dirac operator $\slashed{D}_X$ twisted with the bundle $\mathcal{V}(X)$.
The \emph{Rosenberg index} of $X$ is the class
\begin{equation}\label{eq:A-obstructions}
	\alpha_\CC(X)\ :=\ \ind_A  \slashed{D}_{X,\mathcal{V}(X)}\in K_\ast(C^\ast_\CC\pi)\,,
\end{equation}
where $\ind_A  \slashed{D}_{X,\mathcal{V}(X)}$ is the Mi\v{s}\v{c}enko and Fomenko index of the operator $\slashed{D}_{X,\mathcal{V}(X)}$.


\subsection{Proof of Theorem~\ref{T:Theorem A}}
Suppose the pair $(M,N)$ satisfies the hypotheses of Theorem~\ref{T:Theorem A} and let $g_M$ be a complete Riemannian manifolds on $M$ with $\scal(g_M)>0$.
We want to show that $\alpha_\CC(N)=0$.

We first consider the case when $M$ is odd-dimensional.
In this case, $N$ is even-dimensional and $\alpha_\CC(N)\in K_0(C^\ast_\CC\pi)$.
Let $\overline{M}\rightarrow M$ be the Galois cover such that $\pi_1(\overline{M})=\pi_1(N)$.
There exists a lift of $i:N\hookrightarrow M$ to an inclusion $j:N\hookrightarrow \overline{M}$ such that 
$j_\ast:\pi_1(N)\rightarrow\pi_1(\overline{M})$ is an isomorphism and there is a partition $\overline{M}=\overline{M}_-\cup_{j(N)}\overline{M}_+$ where $j(N)=\overline{M}_-\cap \overline{M}_+$ has codimension one (for more details on this construction, see~\cite[Proof of Theorem~1.7]{Zei16}). 
From Theorem~\ref{T:codimension one obstructions} with the choice $V=\mathcal{V}(M)$, we deduce that $\alpha_\CC(N)=0$.

Let us now consider the case when $M$ is even-dimensional.
In this case, $N$ is odd-dimensional and $\alpha_\CC(N)\in K_1(C^\ast_\CC\pi)$.
Replace the pair $(M,N)$ with the pair $(M\times S^1,N\times S^1)$.
Let $g_{S^1}$ be the canonical flat metric on $S^1$.
Then the product metric $g_M\times g_{S^1}$ on $M\times S^1$ is complete and has positive scalar curvature.
Since $N$ has trivial normal bundle in $M$, then $N\times S^1$ has trivial normal bundle in $M\times S^1$.
Moreover, since $\pi_1(N)$ injects into $\pi_1(M)$, then also $\pi_1(N\times S^1)$ injects into $\pi_1(M\times S^1)$.
Hence, the pair $(M\times S^1,N\times S^1)$ satisfies the hypotheses of Theorem~\ref{T:Theorem A}.
Since $M\times S^1$ is even-dimensional, from the first part of the proof we deduce that $\alpha_\CC(N\times S^1)=0$ in $K_0(C^\ast_\CC\pi)$.
By~\cite[Proposition~4.2]{HPS15}, we finally obtain $\alpha_\CC(N)=0$ in $K_1(C^\ast_\CC\pi)$.
\hfill$\square$

\appendix
\section{Self-adjointness and regularity of $A$-linear differential operators of Schr\"{o}dinger-type.}\label{A:shroedinger-type}
Let  $W$ be a Hilbert $A$-bundle of finite type over a complete Riemannian manifold $(M,g)$.
Let $Q\colon C^\infty_c(M;W)\to C^\infty_c(M;W)$ be a formally self-adjoint, first order, differential operator.
Denote by $\sigma(Q)$ the principal symbol of $Q$.
For $x\in M$ and $\xi\in T^\ast_xM$, the map $\sigma(Q)(x,\xi)$ is a bounded adjointable operator on the Hilbert $A$-module $W_x$.
Let $R:W\rightarrow W$ be a smooth self-adjoint $A$-linear bundle map.
Consider the Schr\"odinger-type operator 
\begin{equation}\label{W:IHR}
       G \ := \ Q^2 \ + \ R\,.
\end{equation}
We view $G$ as an $A$-linear, unbounded operator on $H^0(M;W)$ with initial domain $C^\infty_c(M;W)$.


\begin{theorem}\label{T:G essentially self-adjoint}
Suppose that
\begin{enumerate}[label=(\alph*)]
	\item the operator $Q$ is elliptic;
	\item the principal symbol $\sigma(Q)$ is uniformly bounded from above, i.e. there exists a constant $c>0$ such that
		\begin{equation}\label{eq:Istronglyelliptic}
			\|\sigma(Q)(x,\xi)\|_{\mathcal{L}_A(W_x)} \ \le \ c\,|\xi|_g, \qquad \qquad  
			x\in M, \ \xi\in T_x^*M\backslash\{0\}\,,
		\end{equation}
		where $\|\sigma(Q)(x,\xi)\|_{\mathcal{L}_A(W_x)}$ is the norm of $\sigma(Q)(x,\xi)$ as a bounded operator on $W_x$ 
		and where $|\xi|_g$ denotes the length of $\xi$ defined by the Riemannian metric $g$ on $M$;
	\item the potential $R(x)$ is uniformly bounded from below, i.e. there exists a constant $b>0$ such that
		\begin{equation}\label{eq:R>-b}
			R(x)\ \geq\ -b\,,\qquad\qquad x\in M\,.
		\end{equation}
\end{enumerate}
Then the minimal closure $\overline{G}$ of $G$ is a regular, self-adjoint operator on the Hilbert $A$-module $H^0(M;W)$.
It is the only self-adjoint extension of $G$. 
\end{theorem}

\subsection{The local global principle of Kaad and Lesch}\label{SS:KL}
Let $E$ be a Hilbert $A$-module and let $D:\dom(D)\rightarrow E$ be an $A$-linear, closed, densely defined and symmetric operator.
Let $H_\rho$ be a (not necessarily separable) Hilbert space and let $\rho:A\rightarrow\mathcal{L}(H_\rho)$ be a $\ast$-representation of the $C^\ast$-algebra $A$.
Denote by $E^\rho$ the Hilbert space $E\tensor_\rho H_\rho$, obtained by completing the algebraic tensor product $E\odot_A H_\rho$ with respect to the $\CC$-valued inner product
\[
	\left<e_1\tensor h_1,e_2\tensor h_2\right>_{E^\rho}\ :=\ \left<h_1,\rho\left(\left<e_1,e_2\right>_E\right) h_2\right>_{H_\rho}
\]
(see ~\cite[Section~13.5]{Bla98}).
Let $D^\rho_0:\dom(D)\odot_A H_\rho\rightarrow E^\rho$ be the operator defined by setting $D^\rho_0(e\tensor h)=(De)\tensor h$.
By~\cite[Lemma~2.5]{KL12}, $D^\rho_0$ is densely defined and symmetric.
The closure $D^\rho$ of $D^\rho_0$ is called the \emph{localization} of $D$ at the representation $\rho$.
Notice that the operator $D^\rho$ is closed, densely defined and symmetric.
Finally recall that a representation $\rho:A\rightarrow\mathcal{L}(H^\rho)$ is \emph{cyclic} if there exists $h_0\in H_\rho$ such that the set $\{\rho(a)h_0|a\in A\}$ is dense in $H_\rho$.


\begin{theorem}[Kaad-Lesch,~\cite{KL12}]\label{T:KL}
The following are equivalent:
\begin{enumerate}
\item $D$ is a regular, self-adjoint operator on the Hilbert $A$-module $E$; 
\item for every cyclic representation $\rho$ of $A$, the localization $D^\rho$ is a self-adjoint operator on the Hilbert space $E^\rho$.
\end{enumerate}
\end{theorem}


\subsection{Localization of differential operators: Ebert approach.}\label{SS:Ebert approach}
Let $W$ be a Hilbert $A$-bundle of finite type over a Riemannian manifold $M$.
Let $P: C^\infty_c(M;W)\rightarrow C^\infty_c(M;W)$ be an $A$-linear, formally self-adjoint differential operator of order $p$.
We regard $P$ as an $A$-linear, densely defined, symmetric operator on $H^0(M;W)$ with initial domain $ C^\infty_c(M;W)$.

Let $\rho:A\rightarrow \mathcal{L}(H_\rho)$ be a $\ast$-representation of $A$ and let $E$ be a Hilbert $A$-module.
Then $\rho$ induces a $\ast$-representation $\widehat{\rho}:\mathcal{L}_A(E)\rightarrow \mathcal{L}(E^\rho)$ defined by setting $\widehat{\rho}(P)=P\tensor 1$ (see~\cite[Section~1.5]{Ebe16} and~\cite[Chapter~4]{Lan95}).
Notice that $\widehat{\rho}$ sends unitary operators into unitary operators and induces a continuous group homomorphism $\widehat{\rho}:U(E)\rightarrow U(E^\rho)$.

Let $F$ be the typical fiber of $W$.
Let $\mathcal{P}(W)\rightarrow M$ be a principal $U(F)$-bundle such that $W$ is isometric to $\mathcal{P}(W)\times_{U(F)}F$.
Use the group homomorphism $\widehat{\rho}:U(F)\rightarrow U(F^\rho)$ to construct the bundle $W^\rho :=\mathcal{P}(W)\times_{U(F)}F^\rho$ over $M$ with typical fiber the Hilbert space $F^\rho$.
The completion of $ C^\infty_c(M;W^\rho)$ with respect to the natural $\CC$-valued inner product is a Hilbert space, that we denote by $H^0(M;W^\rho)$.

The operator $P$ extends to a differential operator $P_\rho: C^\infty_c(M;W^\rho)\rightarrow  C^\infty_c(M;W^\rho)$ defined as follows.
Suppose the operator $P$ is given, in local coordinates, by the formula $P=\sum_{|\alpha|\leq q}a_\alpha(x)\partial_x^\alpha$.
Here, for a multiindex $\alpha=(\alpha_1,\ldots, \alpha_d)$ (with $d=\dim M$), we set $|\alpha|=\alpha_1+\cdots +\alpha_d$ and $\partial_x^\alpha=\big(\frac{\partial}{\partial x_1}\big)^{\alpha_1}\cdots\big(\frac{\partial}{\partial x_d}\big)^{\alpha_d}$.
Define $P_\rho$ as the operator defined, in local coordinates, by the formula
\begin{equation}\label{eq:local expression of P_pi}
	P_\rho\ :=\ \sum_{|\alpha|\leq q}\widehat{\rho}(a_\alpha)\,\partial_x^\alpha\,.
\end{equation}
Since $\widehat{\rho}$ is a $\ast$-homomorphism, Formula~\eqref{eq:local expression of P_pi} implies that $P_\rho$ is formally self-adjoint.
Hence, we regard $P_\rho$ as a densely defined, symmetric operator on the Hilbert space $H^0(M;W^\rho)$ with domain $C^\infty_c(M;W^\rho)$.

Let $\overline{P}$ denote the closure of $P$.
Let $\left(H^0(M;W)\right)^\rho$ and $\overline{P}^{\,\rho}$ be resepectively the Hilbert space and the unbounded operator given by the construction of Subsection~\ref{SS:KL}.
We want to relate the Hilbert space $H^0(M;W^\rho)$ with $\left(H^0(M;W)\right)^\rho$ and the operator $P_\rho$ with $\overline{P}^{\,\rho}$.
Define the map
\[
	\Phi: C^\infty_c(M;W)\odot_A H_\rho\longrightarrow  C^\infty_c(M;W^\rho)
\]
by setting $\Phi(s\tensor_A h)(x)=s(x)\tensor_A h\in W_x^\rho$.
It is clear from~\eqref{eq:local expression of P_pi} that $\Phi$ interwines $P_\rho$ and $P\tensor 1$.
We have the following lemma.


\begin{lemma}\label{L:Ebe}
Suppose the representation $\rho$ is cyclic.
Then $\Phi$ extends to an isometry $\Phi:\left(H^0(M;W)\right)^\rho\rightarrow H^0(M;W^\rho)$.
Moreover, $\Phi$ takes $ C^\infty_c(M;W)\odot_A H_\rho$ onto a core of $P_\rho$.
\end{lemma}


\begin{remark}
For the proof of this lemma, we refer the reader to~\cite[Lemma~2.18]{Ebe16}, where the case when $P$ is of first order is proved.
The same argument works in the case when $P$ has arbitrary order $p$ with the following adjustments. 
The $C^1$-norm used in~\cite{Ebe16} must be replaced with the $C^p$-norm and the set $K$ must be defined as $K:=\bigcup_{|\alpha|\leq p}\left(\partial_x^\alpha u\right)(\RR^d)$, where $d=\dim M$.
\end{remark}


\begin{corollary}\label{C:geometric local global principle}
The following are equivalent:
\begin{enumerate}
\item the closure $\overline{P}$ of $P$ is a regular, self-adjoint operator on the Hilbert $A$-module $H^0(M;W)$;
\item for all cyclic representations $\rho$ of $A$, the closure $\overline{P}_\rho$ of $P_\rho$ is a self-adjoint operator on the Hilbert space $H^0(M;W^\rho)$.
\end{enumerate}
\end{corollary}


\begin{proof}
By Lemma~\ref{L:Ebe}, under the isometry $\Phi:\left(H^0(M;W)\right)^\rho\rightarrow H^0(M;W^\rho)$ the operator $\overline{P}^{\,\rho}$ corresponds to $\overline{P}_\rho$.
Now the thesis follows from~Theorem~\ref{T:KL}
\end{proof}


\subsection{The case of closed manifolds}
Let $M$, $W$ and $P$ be as in Subsection~\ref{SS:Ebert approach}.
In this subsection we consider the case when the manifold $M$ is closed and study the unbounded $A$-linear operator $P: C^\infty(M;W)\rightarrow H^0(M;W)$.


\begin{proposition}\label{P:Q self adjoint}
Suppose the manifold $M$ is closed and the operator $P$ is elliptic.
Then the closure $\overline{P}$ of $P$ is a regular, self-adjoint operator on the Hilbert $A$-module $H^0(M;V)$.
It is the unique self-adjoint extension of $P$.
\end{proposition}


\begin{remark}
For the proof of this proposition, we refer the reader to~\cite[pages~6,7]{HPS15}, where the case when $P$ is a twisted Dirac operator is proved.
The argument used in~\cite{HPS15} is based on the pseudodifferential calculus developed by Mi\v{s}\v{c}enko and Fomenko~\cite{FM80} for operators acting on smooth sections of $W$.
The same argument works in the case when $P$ has arbitrary order $p$.
The only two differences are that, in the general case, $P$ has order $p$ in the pseudodifferential calculus and the parametrix of $P$ is an operator $Q$ of order $-p$ such that there are $R$, $S$ of order $-\infty$ satisfying the identities $PQ=1-R$ and $QP=1-S$.
\end{remark}


\begin{corollary}\label{C:Q_rho self adjoint}
Suppose the same hypotheses of Proposition~\ref{P:Q self adjoint} are satisfied and let $\rho$ be a cyclic representation of the $C^\ast$-algebra $A$.
Then the closure of $P_\rho$ is a self-adjoint operator on the Hilbert space $H^0(M;V^\rho)$.
\end{corollary}


\begin{proof}
It follows from Corollary~\ref{C:geometric local global principle} and Proposition~\ref{P:Q self adjoint} 
\end{proof}


\subsection{Essential self-adjointness of $G_\rho$}
In this subsection we present the proof of Theorem~\ref{T:G essentially self-adjoint}.
Let $M$, $W$, $Q$, $R$ and $G=Q^2+R$ be as in Theorem~\ref{T:G essentially self-adjoint}.
Fix a cyclic representation $\rho$ of the $C^\ast$-algebra $A$ and consider the symmetric densely defined operator
\[
	G_\rho=Q_\rho^2+R_\rho: C^\infty_c(M;W^\rho)\longrightarrow H^0(M;W^\rho)
\]
given by the constructions of Subsection~\ref{SS:Ebert approach}.
By Corollary~\ref{C:geometric local global principle}, to prove Theorem~\ref{T:G essentially self-adjoint} we need to show that the operator $G_\rho^\ast$ is self-adjoint.
To this end, we need some information on the asymptotic behavior of sections in $\Dom(G_\rho^\ast)$.

For $j=0,1,2$, let $H^j(M;W^\rho)$ be the $j$-th Sobolev space constructed by using the operator $Q_\rho$.
We denote by $H^j_\text{loc}(M;W^\rho)$ the space of all sections $s\in H^0(M;W^\rho)$ such that $\psi s$ is in $H^j(M;W^\rho)$ for all $\psi\in C^\infty_c(M)$.
From Corollary~\ref{C:Q_rho self adjoint}, it follows that 
\begin{equation}\label{W:domain of H_V^ast is locally square-integrable}
	\Dom\left(G_\rho^\ast\right)\subseteq H^2_\textrm{loc}(M;W^\rho)\,.
\end{equation}
In particular, $Q_\rho s\in H^0_\textrm{loc}(M;W^\rho)$ for $s\in \Dom(H_\mu^\ast)$.

For $\psi\in C^\infty_c(M)$, consider the bundle map $\widehat{Q}_\rho(d\psi):W\rightarrow W$ defined by setting
\begin{equation}\label{W:Q_rho1}
	\widehat{Q}_\rho(d\psi)_x\ :=\ -i\,\widehat{\rho}_x\left(\sigma(Q)(x,d\psi_x)\right)\,,\qquad\qquad x\in M\,,
\end{equation}
where $\sigma(Q)$ is the symbol of $Q$ and where $\widehat{\rho}_x:\mathcal{L}_A(W_x)\rightarrow \mathcal{L}(W_x^\rho)$ is the representation induced by $\rho$ (see Subsection~\ref{SS:Ebert approach}).
Since $\widehat{\rho}_x$ is norm decreasing (see~\cite[page~42]{Lan95}), by~\eqref{eq:Istronglyelliptic} we deduce
\begin{equation}\label{W:Q_rho2}
	\|\widehat{Q}_\rho(d\psi)_x\|\ \leq\ c\,|d\psi_x|_g\,,\qquad\qquad x\in M\,,
\end{equation}
where $c$ is the same constant of Inequality~\eqref{eq:Istronglyelliptic}.
Moreover, by the local expression for $Q$, it follows that $[Q_\rho,\psi]_x=\widehat{\rho}_x([Q,\rho]_x)$.
Hence, 
\begin{equation}\label{W:Q_rho3}
	Q_\rho(\psi s)\ =\ \psi\, Q_\rho s\,+\, \widehat{Q}_\rho(d\psi) s\,.
\end{equation}


Since the Riemannian metric on $M$ is complete, by \cite[Proposition~4.1]{Shu01} there exists a sequence $\{\phi_k\}_{k=0}^\infty$ of compactly supported real-valued smooth functions on $M$ such that 
\begin{itemize}
	\item[\textbf{(C.1)}] $0\leq\phi_k(x)\leq1$, for all $k\in\ZZ_+$ and all $x\in M$; 
	\item[\textbf{(C.2)}] there exists a sequence $\{L_k\}_{k=0}^\infty$ of compact sets exhausting $M$ such that 
		$\phi_k=1$ on $L_K$ and $\supp (\phi_{k})\subseteq L_{k+1}$;
	\item[\textbf{(C.3)}] the sequence $\{d\phi_k\}_{k=0}^\infty$ converges to $0$ in the $\|\cdot\|_\infty$-norm.
\end{itemize}
Notice that, for $s\in\Dom\left(G_\rho^\ast\right)$, $\phi_k\,Q_\rho s\in H^0(M;W^\rho)$ by~\eqref{W:domain of H_V^ast is locally square-integrable}.


\begin{lemma}\label{L:Ds is square-integrable2}
Let $\rho$ be a cyclic representation of the $C^\ast$-algebra $A$.
Suppose that $s\in \Dom\left(G_\rho^\ast\right)$ and that $\{\phi_k\}_{k=0}^\infty$ is a sequence of compactly supported smooth functions satisfying \emph{\textbf{(C.1)}} and \emph{\textbf{(C.3)}}.
Then there is a constant $c_1>0$ such that $\|\phi_k Q_\rho s\|_0\leq c_1$, for all $k\in\ZZ_+$.
\end{lemma}


\begin{proof}
Using~\eqref{W:Q_rho3}, we have 
\begin{equation}\label{W:Ds is square-integrable:2}
\begin{array}{rcl}
	\| \phi_k  Q_\rho s\|_0^2 &=&\left< Q_\rho(\phi_k ^2 Q_\rho s),s\right>_0\ =\ 
	\left<\phi_k ^2Q_\rho^2s,s\right>_0 \,+\,2\left< \widehat{Q}_\rho(d\phi_k )\phi_k  Q_\rho s,s\right>_0\vspace{0.2cm}\\
	&=& 
	\left<\phi_k ^2G_\rho^2s,s\right>_0 \,-\,
	\left<\phi_k^2 R_\rho s,s\right>_0\,+\,2\left< \widehat{Q}_\rho(d\phi_k )\phi_k  Q_\rho s,s\right>_0
\end{array}
\end{equation}
where $\widehat{Q}_\rho(d\phi_k)$ is the bundle map defined by~\eqref{W:Q_rho1}.
Notice that, since $\widehat{\rho}$ is a morphism of $C^\ast$-algebras, it sends positive operators into positive operators.
Since $R_\rho(x)=\widehat{\rho}_x(R(x))$, by~\eqref{eq:R>-b} it follows that $R_\rho(x)\geq -b$ for $x\in M$.
Hence,
\begin{equation}\label{W:Ds is square-integrable:5}
	-\,\left<\phi_k^2 R_\rho s, s\right>_0\ \leq\ b\,\left<\phi_k^2 s, s\right>_0\,.
\end{equation}
From~\eqref{W:Ds is square-integrable:2},~\eqref{W:Ds is square-integrable:5} and~\textbf{(C.1)}, we obtain
\begin{equation}\label{W:Ds is square-integrable:1}
	\| \phi_k  Q_\rho s\|_0^2 \ \leq\  \| G_\rho^\ast  s\|_0 \, \| s\|_0 \,+\,b\, \|s\|_0^2
	\,+\,2\,\|d\phi_k\|_\infty\,\|\phi_k  Q_\rho s\|_0\|s\|_0\,.
\end{equation}
Using the inequality $ab\le \frac12 a^2+\frac12{b^2}$, we get
\begin{equation}\label{W:Ds is square-integrable:3}
	2\,\|d\phi_k\|_\infty\,\|\phi_k  Q_\rho s\|_0\, \|s\|_0 \ \leq\ \frac12\,\|\phi_k  Q_\rho s\|_0^2\ +\  
	2\,\|d\phi_k \|_\infty^2\,\|s\|_0^2\,.
\end{equation}
From~\eqref{W:Ds is square-integrable:1} and~\eqref{W:Ds is square-integrable:3}, we deduce
\[
	\|\phi_k  Q_\rho s\|_0^2\ \leq\ 2\,\| G_\rho^\ast  s\|_0 \, \| s\|_0 \,+\,2 b\, \|s\|_0^2
	\ +\  
	4\,\|d\phi_k \|_\infty^2\,\|s\|_0^2\,.
\]
The thesis follows from this last inequality and~\textbf{(C.3)}.
\end{proof}


\subsection*{Proof of Theorem~\ref{T:G essentially self-adjoint}}
Let $\rho$ be a cyclic representation of the $C^\ast$-algebra $A$ and let $s_1$, $s_2\in \dom(G_\rho^\ast)$.
By Corollary~\ref{C:geometric local global principle}, to prove the thesis we need to show that
\begin{equation}\label{W:P:H_F^ast is regular:6}
	\left< G_\rho^\ast\, s_1, s_2\right>_0\ =\ \left< s_1,G_\rho^\ast\, s_2\right>_0\,.
\end{equation}
Let $\{\phi_k\}_{k=0}^\infty$ be a sequence of compactly supported functions satisfying \textbf{(C.1)--(C.3)}.
We have
\begin{equation}\label{W:P:H_F^ast is regular:1}
\begin{aligned}
	\left<\phi_k   s_1,G_\rho^\ast s_2\right>_0
	=\ &\left< Q_\rho^2(\phi_k   s_1),s_2\right>_0 \,+\,\left< \,R_\rho(\phi_k   s_1),s_2\right>_0\vspace{0.2cm}\\
	=\ &\left< Q_\rho(\phi_k   s_1),Q_\rho s_2\right>_0 \,+\,\left<\phi_k   s_1,R_\rho s_2\right>_0\vspace{0.2cm}\\
	=\ & \left<\phi_k   Q_\rho s_1,Q_\rho s_2\right>_0\,+\,
		\left<\widehat{Q}_\rho(d\phi_k  )s_1,Q_\rho s_2\right>_0+\left<\phi_k   s_1,R_\rho s_2\right>_0\,,
\end{aligned}
\end{equation}
where $\widehat{Q}_\rho(d\phi_k)$ is the bundle map defined by~\eqref{W:Q_rho1}.
Similarly,
\begin{equation}\label{W:P:H_F^ast is regular:2}
	\left< G_\rho^\ast s_1,\phi_k   s_2\right>_0
	\ =\ \left< Q_\rho s_1,\phi_k   Q_\rho s_2\right>_0\,+\,\left< Q_\rho s_1,\widehat{Q}_\rho(d\phi_k)s_2\right>_0\,+\,
		\left< R_\rho s_1,\phi_k   s_2\right>_0\,.
\end{equation}
From (\ref{W:P:H_F^ast is regular:1}) and (\ref{W:P:H_F^ast is regular:2}), we deduce
\begin{equation*}
\begin{aligned}
	\left< G_\rho^\ast s_1,\phi_k   s_2\right>_0\,-\,\left<\phi_k   s_1,G_\rho^\ast s_2\right>_0\ =\ 
	\left< \phi_{k+1}Q_\rho s_1,\widehat{Q}_\rho(d\phi_k)s_2\right>_0\,-\,
	\left<\widehat{Q}_\rho(d\phi_k  )s_1,\phi_{k+1}Q_\rho s_2\right>_0\vspace{0.2cm}\\
\end{aligned}
\end{equation*}
where we used the fact that, by~\textbf{(C.2)}, $\phi_{k+1}=1$ on the support of $\phi_k$.
Therefore, using Lemma~\ref{L:Ds is square-integrable2} we obtain
\begin{equation*}\label{W:P:H_F^ast is regular:11}
	\left|\left< G_\rho^\ast s_1,\phi_k   s_2\right>_0\,-\,\left<\phi_k   s_1,G_\rho^\ast s_2\right>_0\right|\ \leq\
	c_2\,\|d\phi_k\|_\infty\,,
\end{equation*}
for a suitable constant $c_2$.
Hence, 
\begin{equation}\label{W:P:H_F^ast is regular:4}
	\left< G_\rho^\ast s_1,\phi_k   s_2\right>_0\,-\,\left<\phi_k   s_1,G_\rho^\ast s_2\right>_0\,\xrightarrow{\quad\quad} 
	0\,,\ \ \ \ \textrm{as}\ k\rightarrow\infty\,.
\end{equation}
Moreover, from the dominated convergence theorem, we deduce
\begin{equation}\label{W:P:H_F^ast is regular:5}
	\left< G_\rho^\ast s_1,\phi_k   s_2\right>_0\,-\,\left<\phi_k   s_1,G_\rho^\ast s_2\right>_0\,\xrightarrow{\quad\quad} 
	\left< G_\rho^\ast\, s_1, s_2\right>_0\, -\, \left< s_1,G_\rho^\ast\, s_2\right>_0\,,\ \ \ \ \textrm{as}\ k\rightarrow\infty\,.
\end{equation}
Finally, (\ref{W:P:H_F^ast is regular:4}) and (\ref{W:P:H_F^ast is regular:5}) imply (\ref{W:P:H_F^ast is regular:6}).
\hfill$\square$

\end{document}